%% file: delrt-arXiv-v4.tex
\documentclass[11pt]{article}

\pdfoutput=1

\usepackage{lmodern}
\usepackage[T1]{fontenc}
\usepackage[utf8]{inputenc}

\usepackage{setspace}  
\usepackage{authblk}

\usepackage{amsmath}
\usepackage{accents}  
\usepackage{amsxtra}  
\usepackage{amsthm}  
\usepackage{amsfonts}
\usepackage{amssymb}  
\usepackage{esint}  
\usepackage{bm}
\usepackage{bbm}
\usepackage{bbold}  
\usepackage{dsfont}

\usepackage{textcomp}
\usepackage{float}  
\usepackage{flafter}  
\usepackage{afterpage}
\usepackage{caption}

\usepackage[table]{xcolor}  
\usepackage{graphicx}  

\usepackage{array}  
\usepackage{arydshln}
\usepackage{multirow}  

\usepackage[authoryear,round,sort&compress]{natbib}  

\usepackage[bookmarks=false,hyperfootnotes=false,pdfborder={0,0,0},hidelinks=true]{hyperref}
\usepackage{bookmark}

\usepackage{listings}
\usepackage{algorithm}
\usepackage{algorithmic}

\usepackage{soul}

\makeatletter

\input{newtheorems}
\input{newcommands}

\newcommand\Tstrut{\rule{0pt}{2.6ex}}
\newcommand\Bstrut{\rule[-0.9ex]{0pt}{0pt}}

\makeatother

\begin{document}

\title{Hypothesis Testing in the Presence of Multiple
Samples under Density Ratio Models}

\author[1]{Song Cai}
\author[2]{Jiahua Chen}
\author[2]{James V. Zidek}
\affil[1]{School of Mathematics and Statistics, Carleton
University}
\affil[2]{Department of Statistics, University of British
Columbia}
\maketitle


\begin{abstract}
This paper presents a hypothesis testing method 
given independent samples 
from a number of connected populations.
The method is motivated by a forestry project
for monitoring change in the strength of lumber.
Traditional practice has been built upon nonparametric
methods which ignore the fact that
these populations are connected. By pooling the
information in multiple samples through a density ratio
model, the proposed empirical likelihood method
leads to a more efficient inference and therefore 
reduces the cost in applications.
The new test has a classical
chi--square null limiting distribution.
Its power function is obtained under a class of
local alternatives. The local power is found increased even
when some underlying populations are unrelated to the
hypothesis of interest. Simulation studies confirm that this
test has better power properties than potential competitors,
and is robust to model misspecification.
An application example to lumber strength is 
included.
\end{abstract}



\section{Introduction}

The paper presents a method for testing hypotheses 
about parameters of a given number
of different population distributions with 
independent samples from each. The 
method was created as part of a research
program aimed at developing statistical theory
for monitoring change in the 
strength of lumber.  Interest in 
such a program has been sparked by
climate change, which will affect the 
way trees grow, as well by the changing 
resource mix, for example due to increasing 
reliance on plantation lumber.
Added impetus comes from the increasing
importance of wood as 
a construction material due to its sustainability as a
building material.  Moreover, the 
worldwide forest products industry is vast.

Desiderata for the statistical methods used in the long term
monitoring program of lumber includes two key goals.  First
the methods must be efficient to reduce the sizes of the
required samples: testing lumber costs time and money. For
example lumber must be conditioned in the lab over a period
of months before being destructively tested. Toward the goal
of efficiency, this paper proposes a method that borrows
strength across the multiple samples by exploiting an
obvious feature of  the resource, that distinct populations
of lumber over years, species, regions and so on will share
some latent strength characteristics. Second the methods
should ideally be nonparametric in accordance with the
well--ingrained practice in setting standards for forest
products like those in American Society for Testing and
Materials (ASTM) protocols \citetext{\citeauthor{ASTMD1990}
D1990 -- 07}.

These desiderata, lead
to the semiparametric density ratio model (DRM) 
adopted in this paper. More precisely,
suppose we have $m+1$ lumber populations 
with cumulative distribution
functions (CDFs) $F_k(x)$, $k = 0, \, \ldots, \, m$.
We  link them through the DRM assumption:
\begin{align} 
\label{eq:DRM}
  d F_k(x)
  =
  \exp\big\{ \alpha_k + \bbeta_k^\T \q(x) \big\} dF_0(x),
\end{align}
where $x$ could be a single--valued or vector--valued
variable,
$\q(x)$, the \emph{basis function}, is a prespecified
$d$--dimensional function,
and $\btheta_k^\T = (\alpha_k, \bbeta_k^\T)$ are model parameters.
But the baseline distribution $F_0(x)$ is completely
unspecified and for convenience,
we denote $\btheta_0 = \bf 0$.

The DRM is flexible and covers many commonly used distribution families,
including each member of the exponential family.
For example, normal distributions $N(\mu_k, \, \sigma_k)$,
$k=0, \,\ldots, \, m$,
satisfy a DRM with basis function
$\q(x) = {(x, \, x^2)}^\T$ and corresponding parameters
$\beta_k = {(\mu_k/\sigma_k^2 - \mu_0/\sigma_0^2, \,
1/(2\sigma_0^2) - 1/(2\sigma_k^2))}^\T$,
$\alpha_k = \log{(\sigma_0/\sigma_k)} +
\mu_0^2/(2\sigma_0^2) - \mu_k^2/(2\sigma_k^2)$.
Similarly, gamma distributions satisfy a DRM with 
$\q(x) = {(\log x, \, x)}^\T$.
There is also a close relationship between the logistic regression
model in case--control studies and the two--sample DRM
\citep{Qin1997}.

The empirical likelihood (EL) is a natural platform for data analysis
that in recent years has been widely studied in the context of DRM,
\citet{Chen2013} and \citet{Zhang2000} for quantile estimation, 
\citet{Fokianos2004} for density estimation,
and \citet{Keziou2008} the two--sample EL ratio test.
However investigating the properties
of tests constructed under the DRM assumption proves challenging
since the parameters under the null hypothesis 
are often not interior points of the parameter space.
Thus, the limiting distribution of the
EL--based likelihood ratio test cannot be derived from
the usual approach such as the ones given in 
\citet{Owen2001} or \citet{Qin1998}.
Hence instead in this paper, we study the properties of the
\emph{dual empirical likelihood ratio} (DELR) test.
We show that the proposed test statistic has a
classical chi--square null limiting distribution
under fairly general conditions.
We further study its power function under a class of local alternatives
and find that this local power often increases when
additional samples are included in the data analysis even
when their distributions are not related to the hypothesis.
This result supports the use of the DRM for
pooling information across multiple samples.
Under a broad range of distributional settings,
our simulations show that the proposed 
DELR test is more powerful in detecting distributional
changes over samples than many classical tests.
The new method is also found to be model robust: its size
and power are resistant to mild violations to the DRM
assumption.

{\it
An anonymous referee suggested
the semi--parametric proportional
hazards model (CoxPH) proposed by \citep{Cox1972} as 
an alternative for analyzing multiple samples. 
The CoxPH model for multiple sample amounts to 
assuming 
\begin{align*}
h_k(x) = \exp(\beta_k) h_0(x),
\end{align*}
$h_k(x)$ being the hazard function of the $k$th sample.
Clearly, this model would impose a very strong restrictions on  
how $m+1$ populations are connected. 
In comparison, the DRM is much more flexible by allowing the
density ratio to be a function of $x$.
The limitation of the CoxPH approach
for multiple samples 
is easily seen 
in the simulation 
studies whose results  
are included in Sections 5.3 and 5.4.
The power of the partial likelihood ratio test
under the CoxPH is comparable to that of
DRM approach when the proportional hazards assumption
is true. Otherwise, the DRM--based test has
a higher power. 

The CoxPH method 
is superior for multiple populations indexed
by some covariate $z$ when $\exp(\beta_k)$
would be replaced by $\exp(\beta^\T z)$. 
It needs only a single
parameter vector $\beta^\T$.
If the DRM were applied to survival data, 
the conceptual size of $m$ would 
equal the number of distinct $z$ values,
and hence be as large as the sample size.
Thus each of the CoxPH and DRM 
methods have domains of applicability
in which they would be
superior to the other.}


The paper is organized as follows.  
We first review the EL methodology
for multiple samples under the DRM.
We then motivate the use of dual EL to overcome the
associated boundary problem. 
In Section 3, we obtain the limiting distributions
of the DELR statistic under various null hypotheses
and local alternatives.
Section 4 studies the effect of
information pooling on power properties of the DELR test. 
The finite sample properties of
the DELR test are assessed via simulation in Section
5.
An application example to lumber strength is 
given in Section 6.
The simulation details and the proofs are presented in the
Appendices.

\section{EL under the DRM}
\label{sec:EL}

Denote the observations in the $m+1$ samples as
\begin{align*} 
  \{x_{kj}: j=1, \, \ldots, \, n_k \}_{k=0}^{m}
\end{align*}
where $n_k > 0$ is the size of the $k^{th}$ sample. 
We will denote the total sample size as $n=\sum_k n_k$.
Let $\Deriva F_k(x) = F_k(x) - F_k(x^-)$,
and put $p_{kj} = \Deriva F_0(x_{kj})$.
Under the DRM assumption (\ref{eq:DRM}),
the EL of the $\{F_k\}$ is defined to be
\begin{align*}
L_n(F_0, \, \ldots, \, F_m)
  &= \prod_{k, \, j} \Deriva F_{k}(x_{kj})
  = \Big\{ \prod_{k, \, j} p_{kj} \Big\} \cdot
  \exp \Big\{
    \sum_{k, \, j} \big(\alpha_k + \bbeta_k^\T \q(x_{kj})\big)
  \Big\},
\end{align*}
where the sum and product are over all possible $(k, j)$ combinations.
The DRM assumption and the fact that the $\{F_k\}$ are
distribution functions imply that
\begin{align}
\label{eq:drmConstraint}
1
=
\int \Deriva F_k(x)
= 
\int \exp\{\alpha_k + \bbeta_k^\T \q(x)\} \Deriva F_0(x).
\end{align}
Let
$\balpha = (\alpha_1, \, \ldots, \, \alpha_m)^\T$,
$\bbeta^\T = (\bbeta_1^\T, \, \ldots, \, \bbeta_m^\T)$,
and
$\btheta^\T = (\balpha^\T, \, \bbeta^\T)$.
We may also write the EL as $L_n(\btheta, \, F_0)$.

The maximum EL estimator (MELE) of $\btheta$ and $F_0$
is the maximum point of $L_n(\btheta, \, F_0)$
over the space of $\btheta$ and $F_0$ such that
(\ref{eq:drmConstraint}) is satisfied. For both theoretical discussion
and numerical computation, the maximization is carried out in two
steps.
First, we define the \emph{profile log EL}:
\begin{align*}
\tilde{l}_n(\btheta) 
  = \sup \big \{ 
  \log L_n(\btheta, \, F_0): \  
  \sum_{k,\, j} \exp\{ \alpha_r + \bbeta_r^\T \q(x_{kj})\}
    p_{kj} = 1,
  \  r=0, \ldots, m.
  \big \}
\end{align*}
where the supremum is over the space of $F_0$ with fixed $\btheta$.
Based on the method of Lagrange multipliers, the supremum is
found to be attained when
\begin{align}
\label{eq.pp}
p_{kj}
  =
  {n}^{-1} \Big\{ 1 + 
    \sum_{r = 1}^{m} \lambda_r
    \big[\exp \big\{ \alpha_r + \bbeta_r^\T \q(x_{kj}) \big\} - 1\big]
  \Big\}^{-1},
\end{align}
where the Lagrange multipliers $\{\lambda_r\}$ solve,
for $t=0, \, \ldots, \, m$,
\begin{align}
\label{eq.Lag}
  \sum_{k, \, j}
   \exp \big\{ \alpha_t + \bbeta_t^\T \q(x_{kj}) \big\} p_{kj}
   = 1.
\end{align}
The profile log EL can hence be written as
\begin{align*}
  \tilde{l}_n(\btheta)
  = - \sum_{k,\, j} \log \Big\{
    1 + \sum_{r=1}^m \lambda_r
    \big[ \exp \big\{ \alpha_r + \bbeta_r^\T \q(x_{kj}) \big\} - 1\big]
  \Big\} +
  \sum_{k,\, j} \big\{ \alpha_k + \bbeta_k^\T \q(x_{kj}) \big\}.
\end{align*}

The MELE $\hat \btheta$ of $\btheta$ is then the point at which 
$\tilde{l}_n(\btheta)$ is maximized. Given $\hat \btheta$, we solve for
the Lagrange multipliers $\hat \lambda_r$ through (\ref{eq.Lag}).
Interestingly, we always have $\hat \lambda_r = n_r/n$.
Subsequently, we obtain $\hat p_{kj}$ by plugging
$\hat \btheta$ and $\hat \lambda_k$ into \eqref{eq.pp}.
Finally, the MELEs of the $\{F_k\}$ are given by
\begin{align*}
\hat F_k(x)
=
n^{-1} \sum_{r, \, j} 
\exp \big\{ \hat \alpha_k + \hat \bbeta_k^\T \q(x_{rj}) \big\}
\hat p_{rj} 
\ind( x_{rj} \leq x),
\end{align*}
where $\ind(\cdot)$ is the indicator function.

In applications such as that described in the Introduction to the
forestry products industry, 
giving a point estimation is
a minor part of the data analysis. Assessing the uncertainty in the
point estimator and testing hypotheses would 
be judged of 
greater practical importance. 
Asymptotic properties of the point
estimator and the likelihood function enable 
more such in--depth data analyses.
However, classical asymptotic theories usually rely on 
differential properties of the likelihood function in 
the neighbourhood
of the true parameter value.
Consequently these results are applicable only if this
neighbourhood  lies in the parameter space.

According to (\ref{eq:drmConstraint}), $\alpha_k$ is just
a normalizing constant satisfying
\begin{align*}
  \alpha_k
  =
  - \log \int \exp\{\bbeta_k^\T \q(x)\} \Deriva F_0(x).
\end{align*}
Thus, $\alpha_k=0$ whenever $\bbeta_k = \boldsymbol{0}$.
When the true value $\btheta_1 = \boldsymbol{0}$,
its neighborhood will not be contained in the parameter space.
In statistical terminology, DRM is not regular at this $\btheta$, 
as noticed by \citet{Zou2002a}. 
Clearly, the regularity is also violated when
$\bbeta_1 = \bbeta_2$ which implies $\alpha_1 = \alpha_2$. 
In our application, $\btheta_k$ would be
 the parameter of the lumber population
at year $k$ and $\btheta_1 = \btheta_2$ would 
signify the stability
of the wood quality over these two years. Non--regularity
denies a simplistic application of the straightforward
EL ratio test to this important hypothesis.
This creates a need for other effective inferential methods.


\section{Dual EL and its properties}
\label{sec:DEL}

Recall that when $\btheta = \hat \btheta$,
$\hat \lambda_r = n_r/n$. Hence, if we define
\begin{align} \label{eq:del}
  l_n (\btheta)
  = - \sum_{k,\, j} \log \Big\{
    \sum_{r=0}^m \hat \lambda_r
    \exp \big\{ \alpha_r + \bbeta_r^\T \q(x_{kj}) \big\}
  \Big\} +
  \sum_{k,\, j} \big\{ \alpha_k + \bbeta_k^\T \q(x_{kj}) \big\},
\end{align}
then we still have
$
 \hat \btheta = \underset{\btheta}\argmax \, l_n(\btheta)
$.
\citet{Keziou2008} refer to $ l_n (\btheta)$ as the
dual empirical likelihood (DEL) function. 
Compared to the EL under the DRM assumption, the DEL is
well--defined for any $\btheta$ in the corresponding
Euclidean space, has a simple analytical form, and is
concave.
Under a two--sample DRM ($m=1$), \citeauthor{Keziou2008}
found that the corresponding likelihood ratio test statistic
has the usual chi--square limiting distribution for $H_0:
\bbeta_1 = {\bf 0}$. However, this result does not apply to
many hypothesis testing problems of our interest; for
example, there are $m + 1 = 5$ samples and the hypothesis is
\begin{align} \label{hp:complexEx1}
  &H_0: \bbeta_1 = \boldsymbol{0}
  \qquad \text{against} \qquad
  H_1: \bbeta_1 \neq \boldsymbol{0},
\end{align}
where $\bbeta_2$, $\bbeta_3$ and $\bbeta_4$ are nuisance
parameters that do not appear in the hypothesis, or
\begin{align} \label{hp:complexEx2}
  &H_0: \bbeta_1 = \boldsymbol{0}  \text{ and } \bbeta_2 = \bbeta_3
  \qquad \text{against} \qquad
  H_1: \bbeta_1 \neq \boldsymbol{0} \text{ or } \bbeta_2
  \neq \bbeta_3.
\end{align}
These are two problems that we have encountered in our
lumber quality monitoring project.

Many of our inferential problems can be abstractly stated as
testing
\begin{align} \label{hp:composite}
  H_0: \   \g(\bbeta) = {\bf 0}
  \ \ \  \text{against} \ \ \  
  H_1: \  \g(\bbeta) \neq {\bf 0}
\end{align}
for some smooth function
$\g: \mathbb{R}^{md} \to \mathbb{R}^{q}$,
with $q \leq md$, the length of $\bbeta$.
 Recall that $m$ is the number of non--baseline
distributions and $d$, the dimension of the basis
function $\q(x)$.
\emph{We will always assume that $\g$ is thrice
differentiable with a full rank Jacobian matrix
$\partial \g/\partial \bbeta$.}
The parameters $\{\alpha_k\}$ are usually not a part of the
hypothesis, because their values are fully determined by the
$\{\bbeta_k\}$ and $F_0$ under the DRM assumption,
although they are treated as independent parameters
in the DEL.

Let $\tilde \btheta$ be the point at which the maximum of
$l_n(\btheta)$ is attained 
under the constraint $\g(\bbeta) = {\bf 0}$.
The DELR test statistic is defined to be
\begin{align*}
  R_n
  =
  2 \{ l_n (\hat \btheta) - l_n (\tilde \btheta) \}.
\end{align*}
Does $R_n$ have the properties of a regular likelihood ratio test
statistic?
The answer is positive and we state the result as follows, whose proof
is given in the supplementary material.

\begin{theorem} \label{thm:DELRT}
Suppose we have $m+1$ random samples from populations
with distributions of the DRM form 
given in (\ref{eq:DRM})
and a true parameter value $\btheta^*$ such that
$
  \int \exp \{\bbeta_k^\T \q(x)\} \Deriva F_0(x) < \infty
$
for $\btheta$ in a neighbourhood of $\btheta^*$,
$\int \bQ(x)\bQ^\T(x) \Deriva F_0(x)$
is positive definite with $\bQ^\T(x) = (1, \, \q^\T(x))$,
and
$\hat \lambda_k = n_k/n = \rho_k+o(1)$ for some
constant $\rho_k \in (0, 1)$.

Under the null hypothesis $\g(\bbeta)={\bf 0}$,
$
  R_n \to \chi^2_q
$
  in distribution  as $n \to \infty$,  
  where $\chi^2_q$ is a chi--squared random variable
  with $q$ degrees of freedom.
\end{theorem}

When $m=1$ and $\g(\bbeta) = \bbeta_1$, Theorem
\ref{thm:DELRT} reduces to the result of \citet{Keziou2008}.
Theorem \ref{thm:DELRT} covers additional ground.
For instance, it covers the hypthesis testing problems
\eqref{hp:complexEx1} and \eqref{hp:complexEx2}.

The null limiting distribution is most useful for approximating
the p--value of a test but it does not give the power of the test.
For the latter, we use the limiting distribution
of $R_n$ at a local alternative.
Let $\{\bbeta_k^*\}$ be a set of parameter values
which form a null model satisfying \mbox{$H_0: \g(\bbeta) = 0$}
under the DRM assumption.
Let
\begin{align} \label{eq:localAlt}
  \bbeta_k
  =
  \bbeta_k^* + n_k^{-1/2} \bc_k
\end{align}
for some constants $\{\bc_k\}$ be a set of parameter values
which form a local alternative. We denote the distribution functions
corresponding to $\bbeta_k^*$ and $\bbeta_k$ as $F_k$ and $G_k$ with $G_0 =
F_0$, respectively.
Note that the $\{G_k\}$ are placed at $n^{-1/2}$ distance from the $\{F_k\}$.
As $n \to \infty$, the limiting distribution of $R_n$ under this local alternative
is usually non--degenerate and provides useful information on
the power of the test.

Now let
\mbox{$
  U_n
  =
  -{n}^{-1} {\partial^2 l_n(\btheta^*)} / {\partial \btheta \partial \btheta^\T}
$}
for the empirical information matrix.
Its almost sure limit under $H_0$ is a symmetric positive definite
matrix, which may be regarded as an \emph{information matrix} $U$. 
We partition the entries of $U$ in agreement with 
$\balpha$ and $\bbeta$ and represent them as
$U_{\balpha \balpha}$, $U_{\balpha \bbeta}$, $U_{\bbeta \balpha}$
and $U_{\bbeta \bbeta}$.
Let
$\varphi_k(\btheta, \, x) = \exp\{ \alpha_k + \bbeta_k^\T \q(x)\}$,
$k=0, \ldots, m$,
and
\begin{equation} \label{eq:def}
\begin{split}
  \h(\btheta, \, x)
  &=
  ( \rho_1 \varphi_1(\btheta, \, x), \, \ldots, \, \rho_m  \varphi_m(\btheta, \, x)
  )^\T,
  \\ 
  s(\btheta, \, x)
  &=
  \rho_0
  +
  \sum_{k=1}^{m} \rho_k  \varphi_k(\btheta, \, x),
  \\
  H(\btheta, \, x)
  &=
  \diag\{\h(\btheta, \, x)\}
  -
  \h(\btheta, \, x) \h^\T(\btheta, \, x) /{s(\btheta, \, x)}.
\end{split}
\end{equation}
Let $\E_0(\cdot)$ be the expectation operator with respect to $F_0$.
Then, the blockwise algebraic expressions of the
information matrix $U$ in terms of $H(\btheta^*, \, x)$ and $\q(x)$
can be written as
\begin{equation}
\label{eq:Uexpression}
\begin{split}
  U_{\balpha \balpha}
  &=
  \E_0 \big\{H(\btheta^*, \, x)\big\},
  \\
  U_{\bbeta \bbeta}
  &=
   \E_0 \Big\{ H(\btheta^*, \, x)\otimes \big(\q(x) \q^\T(x)\big) \Big\},
  \\
  U_{\balpha \bbeta}
  &=
  U_{\bbeta \balpha}^\T
  =
  \E_0 \big\{H(\btheta^*, \, x)\otimes \q^\T(x)\big\},
\end{split}
\end{equation}
where $\otimes$ is the Kronecker product operator.
We partition the Jacobian matrix of $\g(\bbeta)$ evalueated
at $\bbeta^*$,
$\bigtriangledown = \partial \g(\bbeta^*)/\partial \bbeta$,
into \mbox{$(\bigtriangledown_1, \, \bigtriangledown_2)$},
with $q$ and $md-q$ columns respectively.
Without loss of generality, we assume that
$\bigtriangledown_1$ has a full rank.
Let $I_k$ be an identity matrix of size $k\times k$ and
$
J^\T  =
 (-(\bigtriangledown_1^{-1} \bigtriangledown_2)^\T, \,  I_{md-q})
$.

\begin{theorem} \label{thm:localPower}
  Under the conditions of Theorem \ref{thm:DELRT} and 
local alternative defined by (\ref{eq:localAlt}),
  \begin{align*}
    R_n \to \chi^2_{q}(\delta^2)
  \end{align*}
  in distribution as $n \to \infty$,  
where $\chi^2_q(\delta^2)$ is a non--central chi--square
  random variable with $q$ degrees of freedom and a
  nonnegative non--central parameter
  \begin{align*}
    \delta^2
    =
    \left\{ \begin{array}{cc}
      \etab^\T
      \big\{
        \Lambda
       - \Lambda J \big(J^\T \Lambda J\big)^{-1} J^\T \Lambda
      \big\}
      \etab,
      & \text{if } q < md
      \\
      \etab^\T \Lambda \etab,
      & \text{if } q = md
    \end{array} \right.
  \end{align*}
 where
  $
    \etab^\T
    = 
    ( \rho_1^{-1/2} \bc_1^\T, \,
      \rho_2^{-1/2} \bc_2^\T, \,
      \ldots,
      \rho_m^{-1/2} \bc_m^\T )
  $
  and
  $
  \Lambda
  =
  U_{\bbeta \bbeta} -  U_{\bbeta \balpha} U_{\balpha \balpha}^{-1}
  U_{\balpha \bbeta}$.

  Moreover, $\delta^2 > 0$ except when $\etab$ is in the
  column space of $J$.
\end{theorem}

The proof is given in the supplementary material.
The following example demonstrates one usage of this result:
computing local power of the DELRT test under a given
distributional setting.

\begin{example}[Computing the local power of the DELR test
  for a composite hypothesis]
\label{example:localPow}
Consider the situation where $m + 1 = 3$
samples are from a DRM with
basis function $\q(x) = (x, \, \log x)^\T$,
and the sample proportions are $(0.4, \, 0.3, \, 0.3)$.
Let $F_k$, $k=1,\,2$, be the distributions
with parameters $\bbeta_1^* = (-1, \, 1)^\T$ and
$\bbeta_2^* = (-2, \, 2)^\T$.
Suppose $H_0$ is
\mbox{$\g(\bbeta) = 2\bbeta_1 - \bbeta_2 = 0$}.
Consider the local  alternative 
\begin{align} \label{eq:localAltExample}
  \bbeta_k = \bbeta_k^* + n_k^{-1/2} \bc_k, \text{ for } k = 1, \, 2,
\end{align}
with $\bc_1 = (2, \, 3)^\T$ and $\bc_2=(-1, \, 0)^\T$.

Under the above settings, we find
$\bigtriangledown = (2 I_2, \allowbreak \, - I_2)$ so
$J = ((1/2) I_2, \allowbreak \, I_2)$, and
$\etab \approx (3.65, \allowbreak \, 5.48, \allowbreak \,
-1.83, \allowbreak \, 0)^\T$.
The information matrix $U$ is $F_0$ dependent.
When $F_0$ is $\Gamma(2, \, 1)$,  where
in general $\Gamma(\lambda, \, \kappa)$ 
denotes the gamma distribution  
with shape $\lambda$ and rate $\kappa$,
we obtain the information matrix
\eqref{eq:Uexpression} and hence $\Lambda$,
based on numerical computation. 
We therefore get $\delta^2 \approx 10.29$.

Let $\chi^2_{d, \, p}$ denotes the $p^{th}$ quantile of the
$\chi^2_d$ distribution.
The null limiting distribution of $R_n$ is $\chi^2_2$.  
Thus at the $5\%$ significance level, the null
hypothesis is rejected when $R_n \ge \chi^2_{2, \, 0.95}
\approx 5.99$.
Therefore at the current local alternative, the power of the
DELR test is approximately
$P( \chi^2_2(10.29) \ge 5.99) \approx 0.83$.
%
\end{example}

Theorem \ref{thm:localPower} is also useful for
sample size calculation as demonstrated
in the following example.

\begin{example}[Sample size calculation for Example
  \ref{example:localPow}]
\label{example:localPowSampleSize}

Adopt the settings of Example
\ref{example:localPow}.
Suppose we require the power of the DELR test to be at least
$0.8$ at the alternative of
$\bbeta_1=\bbeta_1^*+(0.5,\, 1.5)^\T$
and $\bbeta_2 = \bbeta_2^* + (0.5, \, 0.5)^\T$
at the $5\%$ significance level.
This alternative corresponds to a local alternative of
the form \eqref{eq:localAltExample} with
with
$\bc_1 = {(0.5\sqrt{n_1}, \, 1.5\sqrt{n_1})}^\T
= 0.5{(\sqrt{0.3 n}, \, 3\sqrt{0.3 n})}^\T$
and
$\bc_2 = {(0.5\sqrt{n_2}, \, 0.5\sqrt{n_2})}^\T
= 0.5{(\sqrt{0.3 n}, \, \sqrt{0.3 n})}^\T$.

Using the above $\bc_1$ and $\bc_2$ and recalling
the sample proportions are $(0.4, \, 0.3, \, 0.3)$,
we obtain
$\etab
= {(0.3^{-1/2} \bc_1^\T, \, 0.3^{-1/2} \bc_2^\T)}^\T
= 0.5\sqrt{n}{(1, \, 3, \, 1, \, 1)}^\T$
as a function of the total sample size $n$.
With the same $J$, $F_0$ and $U$ as obtained in Example
\ref{example:localPow}, and applying the formula given
in Theorem \ref{thm:localPower}, we obtain the non--central
parameter $\delta^2(n)$ as a function of $n$. 
We find that when $n\ge 50$,
\begin{align*}
  \Prob( \chi^2_2(\delta^2(n)) \ge \chi^2_{2, \, 0.95})
  \ge 0.8.
\end{align*}
\end{example}

Moreover, Theorem \ref{thm:localPower} is an effective tool
for comparing the local powers of DELR tests formulated in
different ways. The comparison helps us to determine the
most efficient use of information contained in multiple
samples. The point is discussed in the next section.

\section{Power properties of the DELR test under
the DRM}
\label{sec:NumSample}

Our use of DRM is motivated by its ability to pool
information across a number of samples. We believe the
resulting inferences are more efficient than inferences
based on individual samples. Moreover, strong evidence about
this improved efficiency already exists.
\citet{Fokianos2004} obtained more efficient density
estimators under the DRM than the classical kernel density
estimators based on individual samples; \citet{Chen2013}
found DRM--based quantile estimators to be more efficient
than the empirical quantile estimators. Thus an efficiency
advantage for DRM--based hypothesis tests is anticipated.
This section provides rigorous support for this conjecture.

We adopt the setting posited above for multiple samples from
distributions satisfying the DRM assumption. Yet a
hypothesis of interest may well focus on a characteristic of
just a subset of these populations. If so, why should our
tests be based on all the samples? One answer is found in
their improved local powers as we now demonstrate.

Without loss of generality, consider a null hypothesis
regarding subpopulations $F_0, \, \ldots, \, F_r$
with $r < m$ and let
$
  \bzeta^\T = ( \bbeta_1^\T, \, \ldots, \, \bbeta_r^\T )
$.
The composite hypotheses are specified as
\begin{align} \label{hp:comparison}
  H_0: \ \g( \bzeta) = {\bf 0}
  \ \ \  \text{against} \ \ \  
  H_1: \  \g( \bzeta) \neq {\bf 0}
\end{align}
for some smooth function $\g: \mathbb{R}^{rd} \to \mathbb{R}^{q}$
with $q \le rd$.
A DELR test can be based either on samples from just
$F_0, \, \ldots, \, F_r$, or on the samples from all
the populations $F_0, \, \ldots, \, F_m$.
We denote the corresponding test statistics as $R_n^{(1)}$ and
$R_n^{(2)}$, respectively, for ease of exposition.

Theorem \ref{thm:DELRT} implies that, under the null model
of \eqref{hp:comparison}, $R_n^{(1)}$ and $R_n^{(2)}$ have
the same $\chi^2_{q}$ distribution in the limit. But how do
their asymptotic powers compare to each other? As is well
known, most sensible tests are consistent: the asymptotic
power at any fixed alternative model goes to $1$ as the
sample size $n \to \infty$. Hence, meaningful power
comparisons are often carried out by simulation, or by
assessing their asymptotic powers at local alternatives.

Theorem \ref{thm:localPower} provides a useful tool for the
latter approach. That theorem implies that $R_n^{(1)}$ and
$R_n^{(2)}$ have non--central chi--square limiting
distributions with the same $q$ degrees of freedom, however
with possibly different non--central parameter values at a
local alternative.
By a standard result in distribution theory
\citep[(29.25a)]{Johnson1995}, if two non--central
chi--square distributions have the same degrees of freedom,
then the one with the greater non--central parameter
stochastically dominates the one with the smaller
non--central parameter.
Therefore, a power comparison of $R_n^{(1)}$ and $R_n^{(2)}$
can be carried out by comparing their corresponding
non--central parameters $\delta_1^2$ and $\delta_2^2$: if
$\delta_2^2 > \delta_1^2$, then  $R_n^{(2)}$ is more
powerful than $R_n^{(1)}$ at all significance levels, and
vice versa. The following theorem, whose proof is given in
the supplementary material,
implements this idea and provides that power comparison.

\begin{theorem} 
\label{thm:pwComp}
Adopt the conditions of Theorem \ref{thm:DELRT}.
Consider the composite hypothesis \eqref{hp:comparison}
and the local alternative
\begin{align} \label{eq:localAltCompare}
  \bbeta_k
  =
  \left\{ \begin{array}{cl}
  \bbeta_k^* + n_k^{-1/2} \bc_k,
  &\text{ for } k = 1, \ldots, \, r
  \\
  \bbeta_k^*,
  &\text{ for } k=r+1, \ldots, \, m
  \end{array}\right.
\end{align}
with some given constants $\{ \bc_k \}$.
Let $\delta_1^2$ and  $\delta_2^2$ be non--central
parameter values of the limiting distribution of
$R_n^{(1)}$ and $R_n^{(2)}$ under the local alternative
model. Then  $\delta_2^2 \ge \delta_1^2$.
\end{theorem}

\begin{example}[Effect of information pooling by DRM on the
  local power of the DELR test]
\label{example:powComp}
Consider the situation where $m + 1 = 3$, samples are 
from a DRM with basis function $\q(x) = {(x, \, x^2)}^\T$,
and the sample proportions are $(0.5, \, 0.25, \, 0.25)$.
Let $F_k$, $k=1,\,2$, be the distributions
with parameters $\bbeta_1^* = (6, \, -1.5)^\T$ and
$\bbeta_2^* = (-0.25, \, 0.375)^\T$.
Suppose $H_0$ is given by
$\g(\bzeta) = \bbeta_1 - (6, \, -1.5)^\T = \boldsymbol{0}$,
and the local alternative is
\begin{align*}
\bbeta_1 = \bbeta_1^* + n_1^{-1/2} \bc_1;
~~
\bbeta_2 = \bbeta_2^*
\end{align*}
with  $\bc_1 = (2, \, 2)^\T$.

Let $R_n^{(1)}$ and $R_n^{(2)}$ be the DELR test statistics
based on $F_0, F_1$,
and on $F_0, F_1, F_2$, respectively.
When $F_0$ is, $N(0, \, 1)$, the standard normal distribution, we
obtain information matrices \eqref{eq:Uexpression}, and hence
$\Lambda = U_{\bbeta \bbeta} -
U_{\bbeta \balpha} U_{\balpha \balpha}^{-1} U_{\balpha \bbeta}$,
for $R_n^{(1)}$ and $R_n^{(2)}$ based on numerical computation.
For $R_n^{(1)}$, we get
$\etab = (4, \, 4)^\T$ and $q = d =2$.
Then by Theorem \ref{thm:localPower}, we find
$\delta_1^2 = \etab^\T \Lambda \etab \approx 5.90$.
For $R_n^{(2)}$,
we find $\etab = (4, \, 4, \, 0, 0)^\T$,
$\bigtriangledown = (I_2, \, 0_{2\times2})$,
and $J = (0_{2\times2}, \, I_2)^\T$.
By Theorem \ref{thm:localPower},
we get $\delta_2^2 \approx 6.67$.
Now, since $\delta_1^2 \approx 5.90 < \delta_2^2 \approx 6.67$,
$R_n^{(2)}$ is more powerful than $R_n^{(1)}$ even though the null
hypothesis concerns the parameter of just population $1$.

At the $5\%$ significance level, the powers of $R_n^{(1)}$
and $R_n^{(2)}$ are approximately $0.577$ and $0.633$,
respectively.
\end{example}

\section{Simulation studies}
\label{sec:Simulation}

We conducted simulations to study: (1) the approximation accuracy of
the limiting distributions to the finite--sample distributions of the
DELR statistic under both the null and the alternative models, (2) the
power of the DELR test under correctly specified and also misspecified
DRMs, and (3) the effect of the number of samples used in the DRM to
the local asymptotic power of the DELR test.
The number of simulation runs is set to $10,000$.
Our simulation is more extensive than what are presented
in terms of hypothesis, population distribution, and sample sizes.
We selected the most representative ones and included them 
here; but the other results are similar.
All computations are carried out by our R package
{\emph{drmdel}}
for EL inference under DRMs, which  is available on the Comprehensive
R Archive Network
(CRAN).

\subsection{Approximation to the distribution of the
DELR under the null model}

We first study how well the chi--square distribution
approximates the finite--sample distribution of the DELR statistic under the
null hypothesis of (\ref{hp:composite}).
Set $m + 1 = 6$ and consider the hypothesis with
$\g(\bbeta) = (\bbeta_1^\T, \, \bbeta_3^\T) - (\bbeta_2^\T, \, \bbeta_4^\T)$.
The null hypothesis is
equivalent to $F_1 = F_2$ and $F_3 = F_4$. 
We generate two sets of six samples of sizes
$(90, \, 60, \, 120, \, 80, \, 110, \, 30)$
from two different distribution families, respectively.
The first set of samples are from normal distributions with means
$(0, \allowbreak \, 2, \allowbreak \, 2, \allowbreak \, 1,
\allowbreak \, 1, \allowbreak \, 3.2)$
and standard deviations
$(1, \allowbreak \, 1.5, \allowbreak \, 1.5, \allowbreak \, 3,
\allowbreak \, 3, \allowbreak \, 2)$.
The second set of samples are from gamma distributions with shapes
$(3, \allowbreak \, 4, \allowbreak \, 4, \allowbreak \, 5,
\allowbreak \, 5, \allowbreak \, 3.2)$
and rates
$(0.5, \allowbreak \, 0.8, \allowbreak \, 0.8, \allowbreak \, 1.1,
\allowbreak \, 1.1, \allowbreak \, 1.5)$.

When the basis function $\q(x)$ is correctly specified,
i.e. $\q(x) = {(x, \, x^2)}^\T$ for the normal family
and $\q(x) = {(\log x, \, x)}^\T$ for gamma family,
the DELR statistic, $R_n$, has a $\chi^2_4$ null limiting distribution
in both cases. 
The quantile--quantile (Q--Q) plots of the distribution of $R_n$ and
$\chi^2_4$ are shown in Figure \ref{fig:null}. 
In both cases, the approximations are very accurate. 
The type I error rates of $R_n$ at $5\%$ level
are 0.056 and 0.058 for normal and gamma data respectively.

\begin{figure}[!ht]
\begin{center}
    \includegraphics[scale=.35]{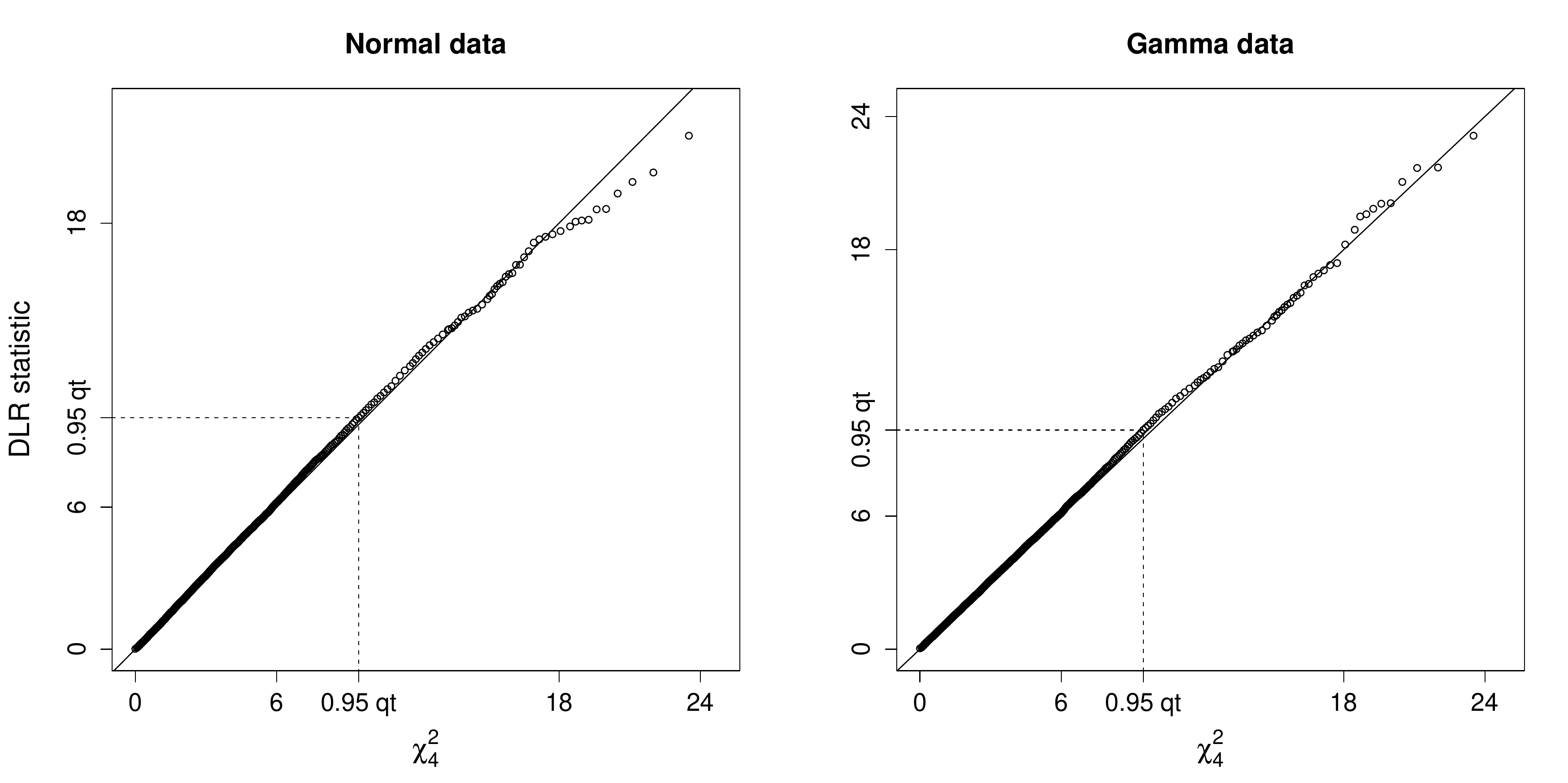}
    \caption{Q--Q plots of the simulated and the
    null limiting distributions of the DELR statistic.}
    \label{fig:null}
\end{center}
\end{figure}

In unreported extensive simulation studies under various
settings, we find that in general
the chi--square approximation has satisfactory precision when 
$n_k \geq 10qd$.
When $n_k$ is much smaller, a bootstrap or permutation test
based on the DELR statistic can be served as an alternative.

\subsection{Approximation to the distribution of the
DELR under local alternatives}

We next examine the precision of the non--central chi--square
distribution under the local alternative model (\ref{eq:localAlt}).
We set $m + 1 = 4$ with sample sizes  $120$, $160$, $80$ and $60$.

In the first scenario, we test the hypothesis (\ref{hp:composite}) with
$\g(\bbeta) = \bbeta_1^\T - \bbeta_2^\T$.
The perceived null model is specified by 
$\bbeta_1^* = \bbeta_2^* = (0.25, \, 1.875)^\T$,
$\bbeta_3^* = (0.125, \, 1.97)^\T$ 
with basis function $\q(x) = (x, x^2)^\T$.
The data were generated from $G_0 =N(0, \, 0.5^2)$, $G_1$ and $G_3$
with $\bbeta_1^*$ and $\bbeta_3^*$ respectively,
and $G_2$ with $\bbeta_2 = \bbeta_2^* + n_2^{-1/2}(1, \, 0)^\T$.
According to Theorem \ref{thm:localPower}, the limiting distribution of 
$R_n$ is $\chi_2^2(2.67)$.

In the second scenario, we test (\ref{hp:composite}) with
$\g(\bbeta) = (\bbeta_1^\T, \, \bbeta_3^\T) - (\bbeta_2^\T, \, (-6, \, 9)^\T)$.
The perceived null model is specified by $\bbeta_1^* = \bbeta_2^* = (-4, \, 5)^\T$, 
$\bbeta_3^* = (-6, \, 9)^\T$ with basis function $\q(x) = (\log x, x)^\T$.
We generated data from $G_0 = \Gamma(3, \, 2)$ and $G_k$, $k = 1, \, 2, \, 3$,
specified by (\ref{eq:localAlt}) with $\bc_1=(0.5, \, 0.5)^\T$, 
$\bc_2 = (1, \, 1)^\T$ and $\bc_3 = (2, \, 2)^\T$.
According to Theorem  \ref{thm:localPower}, the limiting
distribution of $R_n$ is $\chi_4^2(1.80)$.

The Q--Q plots under the two scenarios are shown in Figure \ref{fig:alt}.
It is clear the non--central chi--square limiting distributions approximate
these of of $R_n$ very well.
In unreported simulation studies under various settings, we find the
approximation of the non--central chi--square is generally
satisfactory when $n_k \geq 15qd$.

\begin{figure}[!ht]
\begin{center}
    \includegraphics[scale=.35]{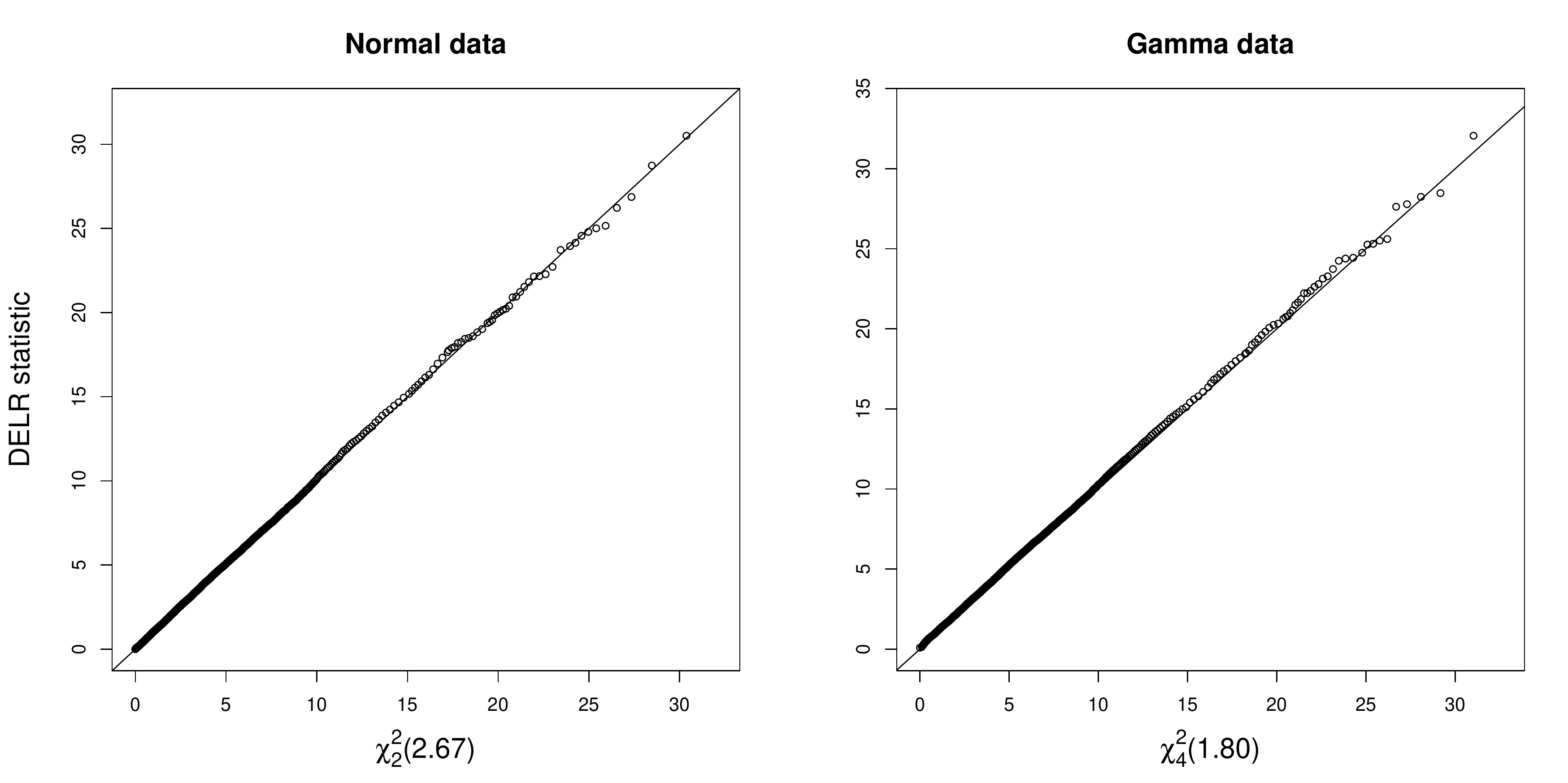}
    \caption{Q--Q plots of the distributions of the DELR statistics under the
      local alternative model against the corresponding asymptotic theoretical
      distributions.}
    \label{fig:alt}
\end{center}
\end{figure}

\subsection{Power comparison}
\label{subsec:pwComp}

We now compare the power of the DELR test (DELRT) with a
number of popular methods for detecting differences between
distribution functions, testing $H_0: F_0 = F_1 = \ldots = F_m$.
This is the same as (\ref{hp:composite}) with $\g(\bbeta) = \bbeta$.
We use the nominal level of 5\%.

The competitors include the Wald test based on
DRM (Wald) \citep[(17)]{Fokianos2001}, 
one--way analysis of variance (ANOVA),
the Kruskal--Wallis rank--sum test (KW) \citep{Wilcox1995},
the k--sample Anderson--Darling test (AD)
\citep{Scholz1987}, and the likelihood ratio test based on
the partial likelihood under the CoxPH
when observations are intrinsically positive.


The Wald test is based on test statistic
$n \hat \bbeta^\T \hat \Sigma^{-1} \hat \bbeta$
with $\hat \Sigma$ being a consistent estimator of the
asymptotic covariance matrix of $\hat \bbeta$. It uses a
chi--square reference distribution. KW is a rank--based
nonparametric test for equal population medians. AD is a
nonparametric test based on the quadratic distances of
empirical distribution functions for equal population
distributions. 
Under the CoxPH, we utilized $m$ dummy covariates
for data analysis. The corresponding 
likelihood ratio based on the partial
likelihood has a $\chi^2_m$ limiting distribution.


We first compare their powers based on normal data with $m +
1 = 2$ and sample sizes $n_0=30$ and $n_1=40$. We consider
two different scenarios for alternatives both having $F_0 =
N(0, \, 2^2)$. In the first scenario, $F_1= N(\mu, 2^2)$
with $\mu$ increasing in absolute value in a sequence of
simulation experiments.
In the second scenario, we consider seven parameter settings
(settings 0--6)
for $F_1 = N(\mu, \sigma^2)$
with $\mu$ and $\sigma$
taking values in 
$(0, \allowbreak \, 0.05, \allowbreak \, 0.1, \allowbreak \, 0.15,
\allowbreak \, 0.25, \allowbreak \, 0.36, \allowbreak \, 0.55)$
and
$(2, \allowbreak \, 1.9, \allowbreak \, 1.8, \allowbreak \, 1.7,
\allowbreak \, 1.62, \allowbreak \, 1.56, \allowbreak \, 1.50)$
respectively.

The power curves are shown in Figure \ref{fig:pwNormal}. In
the two--sample case, ANOVA reduces to the two--sample t--test
and the KW reduces to the Wilcoxon rank--sum test (Wilcoxon).
Yet all tests are found to have comparable powers.
It is against the common sense that the two--sample t--test is
most powerful and the Wilcoxon test is inferior.
In fact, \citet[3.4]{Lehmann1999} found that for normal populations,
the relative efficiency the Wilcoxon test to the
t--test is $3/\pi \approx 0.955$. 
In the unequal variance scenario, the DELR test clearly has
much higher power than its competitors, and its type I error rate
is close to the nominal 0.05.

\begin{figure}[!ht]
\begin{center}
    \includegraphics[scale=.35]{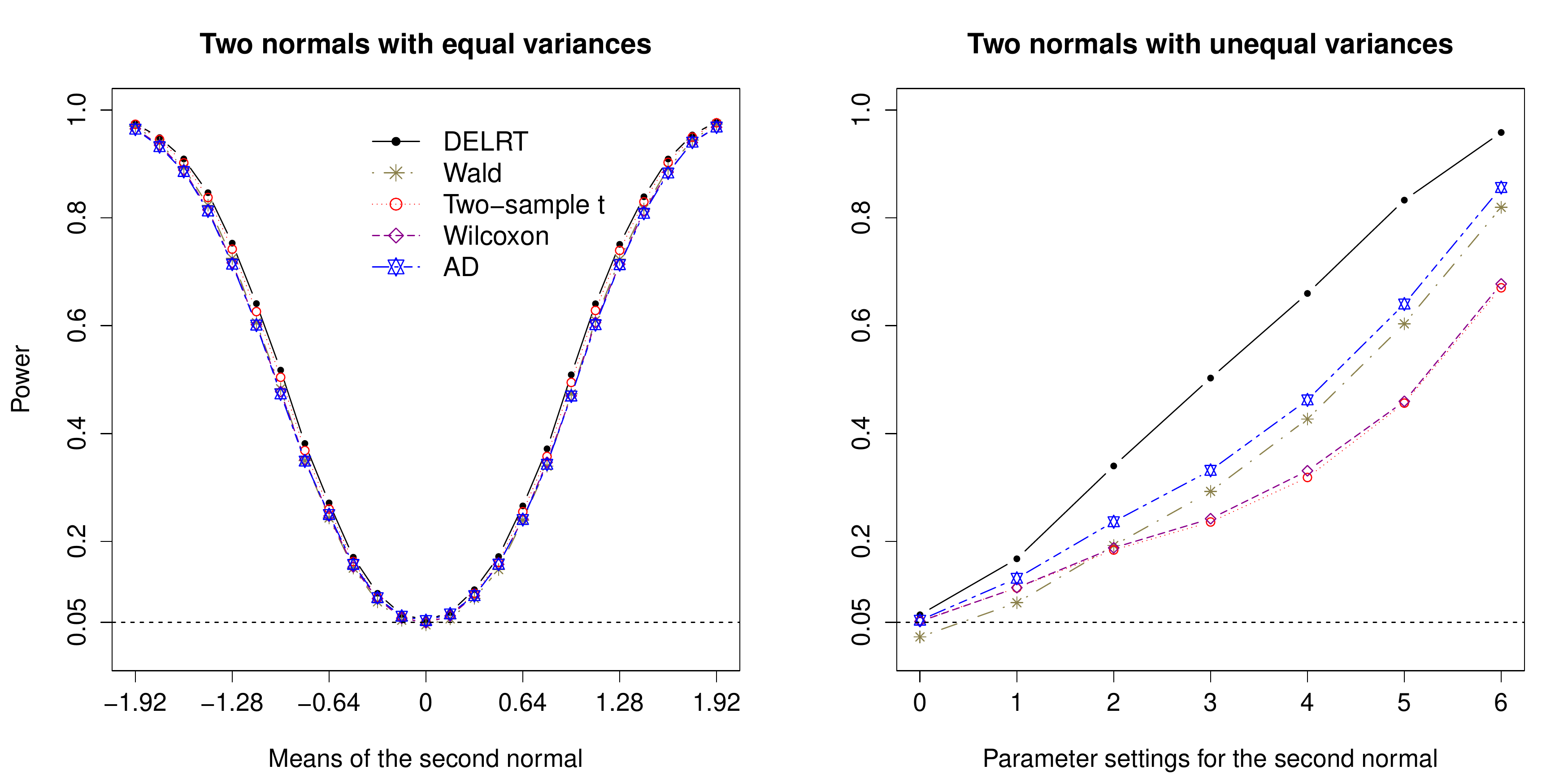}
    \caption{Power curve of $R_n$ under normal data; 
    the parameter setting $0$ corresponds to the null model and the
     settings 1--6 correspond to alternative models.}
    \label{fig:pwNormal}
\end{center}
\end{figure}

We next compare these tests on non--normal samples with
$m + 1 = 5$ and sample sizes to be $30$, $40$, $25$, $45$ and $50$.
We generated data from four families of distributions: 
gamma, log--normal, Pareto with common
support, and Weibull distributions with shape parameter
equaling $0.8$, respectively.
The log--normal, Pareto and Weibull distributions satisfy DRMs with basis
functions
$\q(x) = {(\log x, \, \log^2 x)}^\T$,
$\q(x) = \log x$, 
and $\q(x) = x^{0.8}$, respectively.

For each distribution family, we obtain simulated power
under six DRM parameter settings
(settings 0--5; shown in Table \ref{tab:pwcomp} in the
Appendix II).
Setting 0 satisfies the null hypothesis and settings 1--5 do
not.
The simulated rejection rates are shown in Figure \ref{fig:pwNonNormal}.
It is clear that the DELR test has the highest
power while its type I error rates are close to the nominal.


We note that the gamma and log--normal families do not satisfy the conditions needed to justify use 
of the CoxPH approach.
Consequently, the DELR test based on the DRM has a much
higher power than the likelihood ratio test based on the
partial likelihood under the CoxPH model.
In contrast, the Pareto or Weibull families with 
known, common shapes
do satisfy the CoxPH requirements; in these cases, the two tests have almost
the same power. 
These results show that in general the DRM is a better
choice for multiple samples.


\begin{figure}[!ht]
\begin{center}
    \includegraphics[scale=.35]{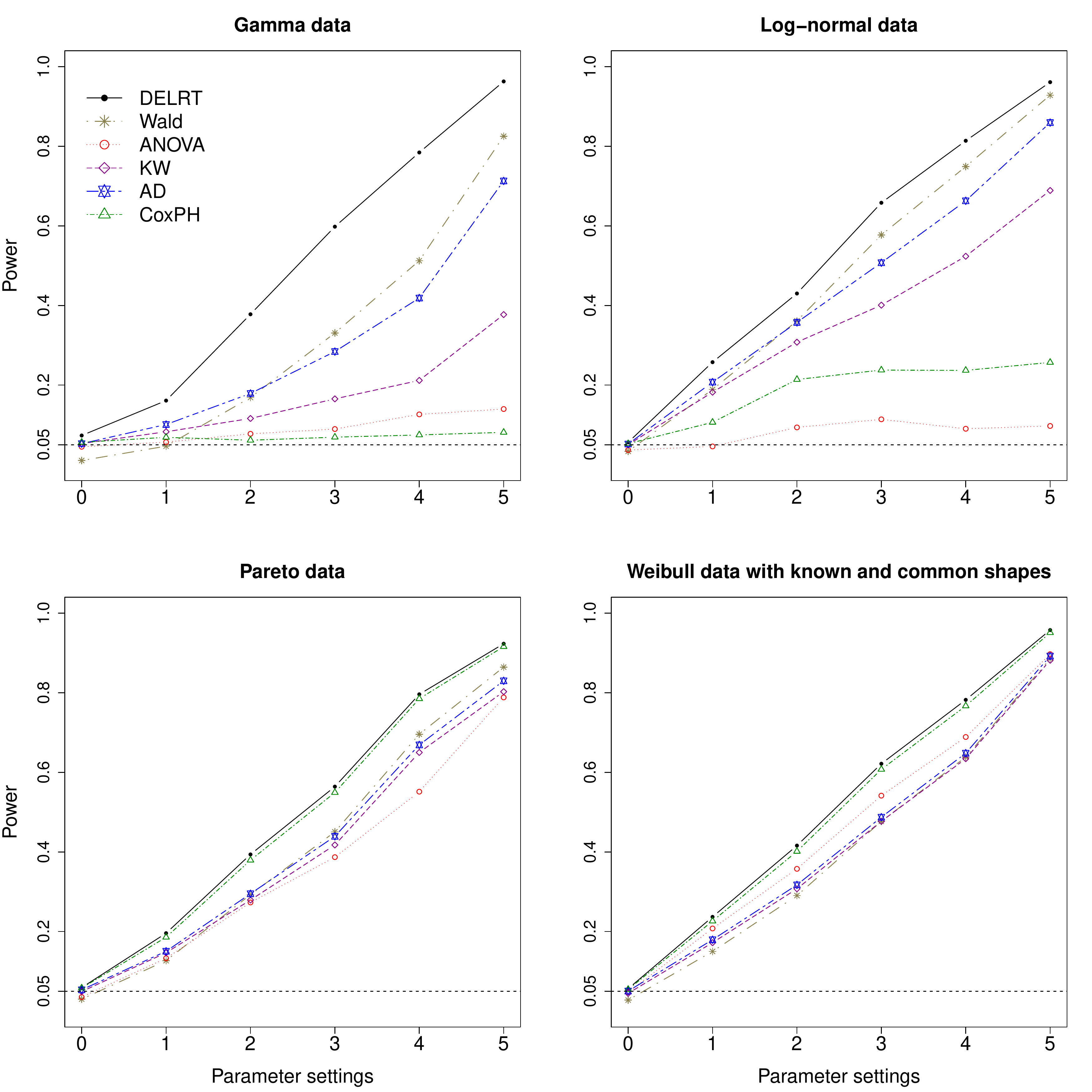}
    \caption{Power curves that obtain
    when the population distribution 
    from which the data are sampled 
    is non-normal;
      the parameter setting $0$ corresponds to the null model and the
     settings 1--5 correspond to alternative models.}
    \label{fig:pwNonNormal}
\end{center}
\end{figure}

\subsection{DELR test under misspecified DRM}
\label{subsec:pwCompMis}

The DRM is very flexible and includes a large number of
distribution families as special cases. The risk of misspecification
is low, and even lower when a high dimensional
basis function $\q(x)$ is utilized. Nevertheless,
examining the effect of misspecification remains an important
topic. \citet{Fokianos2006} suggested that misspecifying the 
basis function $\q(x)$ has an adverse
effect on estimating $\bbeta$.
\citet{Chen2013} found that estimation of population quantiles
is robust against misspecification. In this section,
we demonstrate that the effect of misspecification on DELR test is
small for testing equal population hypothesis.

We put $m+1=5$ with sample
sizes  $90$, $120$, $75$, $135$ and $150$.
In the first simulation experiment, we generated data from
two--parameter Weibull distributions, whose density
function is given by
\begin{align*}
  f(x; \, a, \, b) = (a/b) (x/b)^{a-1} \exp\{(-x/b)^a\},
  \  x \ge 0,
\end{align*}
where $a$ and $b$ are called the shape and scale parameters,
respectively.
The log density ratio of two Weibull distributions
is not linear in known functions of $x$, when they have
unknown $a$ values.
Hence two--parameter
Weibull family does not satisfy the DRM assumption
\eqref{eq:DRM}.
Nevertheless, we still fit a DRM with
$\q(x)={(x, \, \log x)}^\T$ to the Weibull data.
Clearly, this DRM is misspecified.
We use DELR test and Wald test under this DRM to test the
equal distribution function hypothesis.
We calculate the simulated power of these tests under six
parameter settings (Table \ref{tab:pwcompmis} in the
Appendix II) with the setting $0$ satisfying the null
hypothesis.

%

We also apply ANOVA, KW, AD, and the CoxPH. The results are
summarized as power curves in Figure
\ref{fig:pwNonNormalMis}. We notice that the DELR test has
close to nominal type I error rates. It has superior power
in detecting distributional differences. 
In particular, our
DELR approach has a much higher power than the 
CoxPH.
Note that the two--parameter Weibull distributions do not have
proportional hazards. 


\begin{figure}[!ht]
\begin{center}
    \includegraphics[scale=.35]{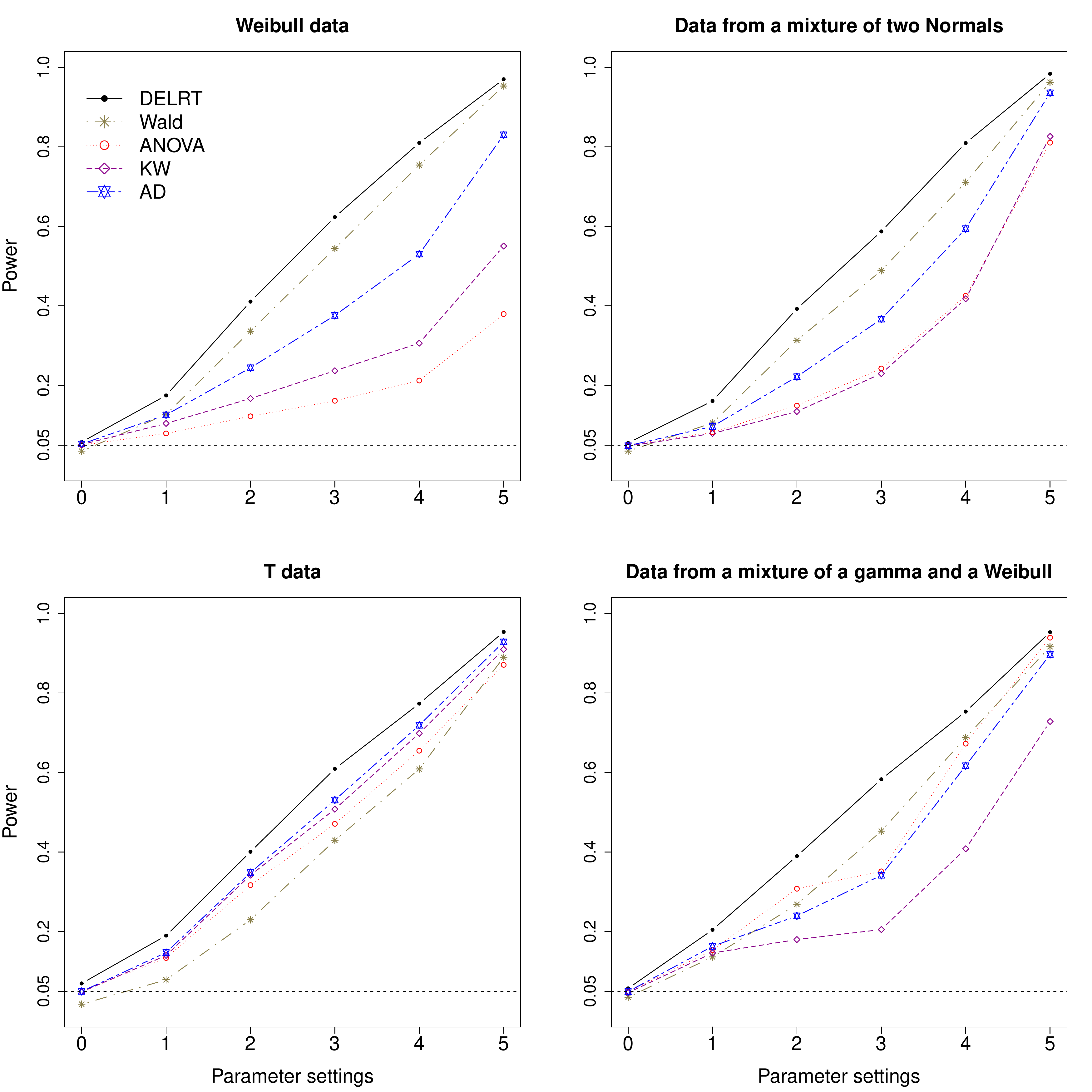}
    \caption{Power curves of five tests. 
      Parameter setting $0$ corresponds to the null model;
      settings 1--5 form alternative models.}
    \label{fig:pwNonNormalMis}
\end{center}
\end{figure}

We have also experimented with other models where the DRM
assumption is violated such as the mixture of two normals,
the non--central t, and the mixture of a gamma and a
Weibull. The results are similar to that for two--parameter
Weibull family.

\subsection{Comparison of $R_n^{(1)}$ and $R_n^{(2)}$}
\label{subsec:DRM1vs2}

Is it helpful to have data from other populations in DELR analysis?
In Theorem 3, we defined $R_n^{(1)}$ and $R_n^{(2)}$ 
and obtained a positive answer in Example 3.
In this section, we reaffirm this conclusion by means
of a simulation study.
Because the same question can be asked about the Wald tests,
we similarly define Wald$^{(1)}$ and
Wald$^{(2)}$ and include in our simulation study.
The number of simulation repetitions is set to $10,000$.


The first simulation uses the setting
in Example \ref{example:powComp}, where
data are from $m+1=3$ normal populations 
with the null hypothesis being $\bbeta_1 = (6, \, -1.5)^\T$.
The total sample size $n$ is $240$.
We calculated the powers of
$R_n^{(1)}$, $R_n^{(2)}$, Wald$^{(1)}$ and Wald$^{(2)}$
with the six different DRM parameters 
as shown in the Appendix II as the 
``Normal Case''  in Table
\ref{tab:pwcomp1vs2}.
The simulated power curves are in Figure \ref{fig:DELRT2vs1} (a).
The power comparisons
between $R_n^{(1)}$ and $R_n^{(2)}$ 
yield conclusions that would have 
been predicted by the conclusions 
of Theorem \ref{thm:pwComp}.
The Wald tests are not as powerful as the
DELR tests, but 
Wald$^{(2)}$ does seem 
to be more powerful than Wald$^{(1)}$.

Even if the additional samples are from distributions 
not under comparison, they may well 
be helpful in estimating 
the baseline distribution $F_0$. 
If so, we would be better 
able to identify the
differences among the distributions under comparison. To
explore these heuristic conclusions
we conducted the following simulation. 

Let $m + 1 = 4$ and consider a hypothesis test for $\bbeta_1$. 
The DELR test can be done using the
first two samples ($R_n^{(1)}$) and then using
all four samples
($R_n^{(2)}$).

We generated samples with sizes $60$, $30$ 
$40$ and $90$
from gamma 
distributions under two scenarios. In the first
scenario, the extra populations $F_2$ and $F_3$
are close to $F_0$.
Because of this, the samples from $F_2$ and $F_3$
are particularly helpful at accurately estimating $F_0$.
In the second scenario, $F_2$ and $F_3$ are rather
distinct from $F_0$.
Because of this, the samples from $F_2$ and $F_3$
are less helpful at estimating $F_0$,
The density functions of $F_0$, $F_2$ and $F_3$ along with
their parameter values under both scenarios are depicted  in
Figure \ref{fig:DELRT2vs1} (b).

Under both scenarios, we consider the same null hypotheses of
$\bbeta_1 = (-2, \, 2)^\T$. We simulated the powers of the
tests at six different values of $\bbeta_1$
(See the ``Gamma Case'' in 
Table \ref{tab:pwcomp1vs2} of the Appendix II) and we simulated power curves are
shown in Figure \ref{fig:DELRT2vs1} (c) and (d).
The degrees of improvement of $R_n^{(2)}$ under two scenarios
clearly match our intuition.
The same phenomenon is also evident for the Wald test.

\begin{figure}[!ht]
\begin{center}
    \includegraphics[scale=.35]{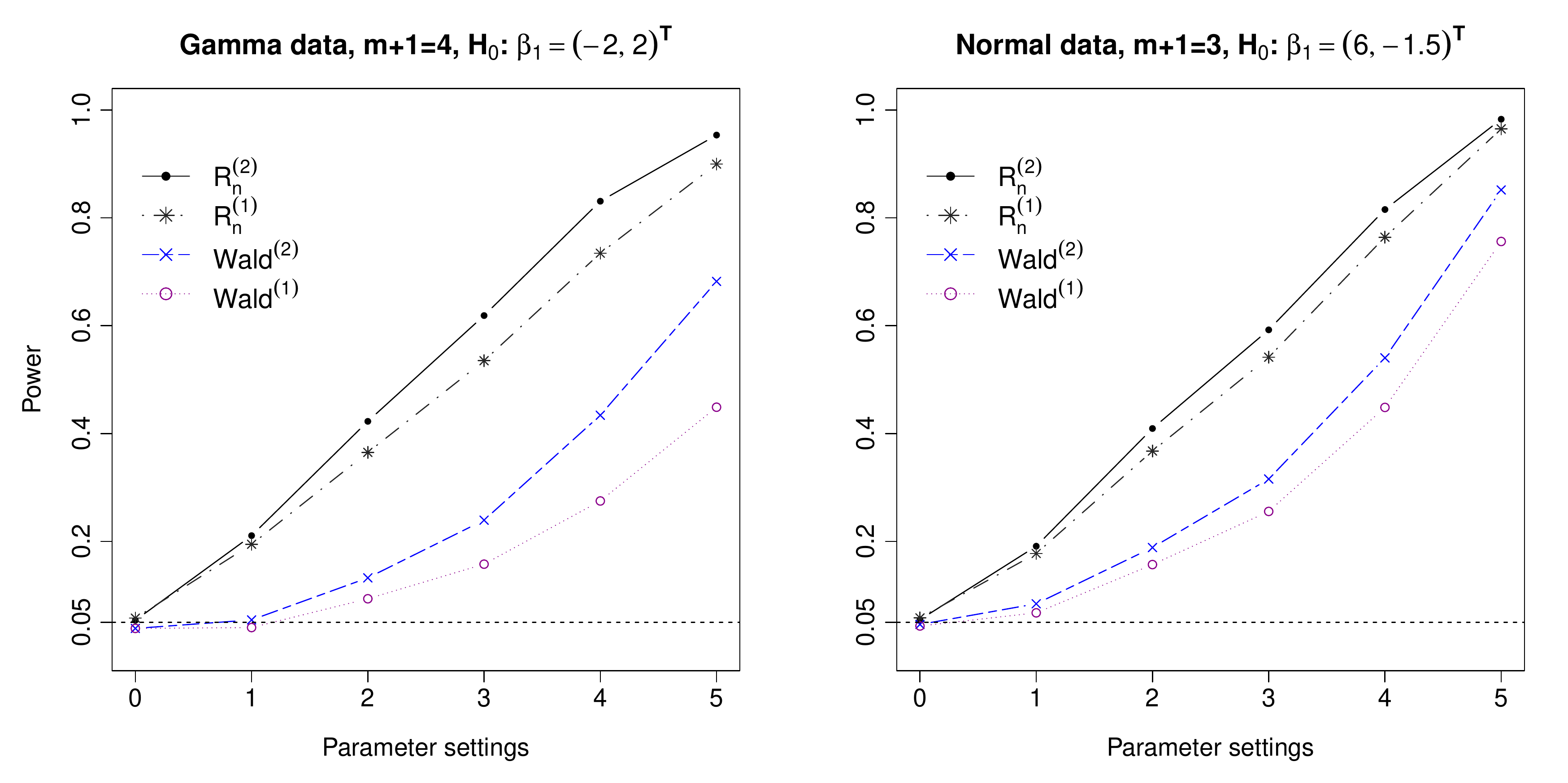}
    \caption{Power curves of $R_n^{(1)}$, $R_n^{(2)}$, Wald$^{(1)}$ and Wald$^{(2)}$;
      Parameter setting $0$ corresponds to the null model;
      settings 1--5 correspond to alternative models.}
    \label{fig:DELRT2vs1}
\end{center}
\end{figure}

\subsubsection{Effects of the length of the basis
function on DELR tests}

We cannot guarantee that the additional populations
are exactly the tilts of $F_0$ with a specific basis function. 
This problem can be alleviated by expanding the basis function
so that good approximations are ensured. 
Expanding the basis function, however, may have an adverse
effect on the power. We investigate this issue here.

Let $m + 1 = 4$ and consider a hypothesis test for
$\bbeta_1$ as in the last simulation. Again we compare the
tests based on the first two samples and the ones based on
all four samples.

We adopt the same distribution and parameter settings for
$F_0$ and $F_1$ as in the last simulation (parameter values
shown in Table \ref{tab:pwcomp1vs2} ``Gamma Case'' in the
Appendix II). However, we set $F_2$ to be log--normal with mean
$0$ and standard deviation $1$ on log scale and $F_3$ to be
Weibull with shape $2$ and scale $3$.
We consider the following four basis functions:
(1) $\q(x) = (\log x, \, x)^\T$,
(2) $\q(x) = (\log x, \, \sqrt{x}, \, x)^\T$,
(3) $\q(x) = (\log x, \, \sqrt{x}, \, x, \, x^2)^\T$,
and
(4) $\q(x) = (\log x, \, \sqrt{x}, \, x, \, x^{1.5}, \, x^2)^\T$.

The simulation results are shown in Figure \ref{fig:DELRT2vs1Basis}. 
Note that the parameter setting 0 corresponds to the null model.
We have the following observations:
\begin{enumerate}
\item
$R_n^{(2)}$ is more powerful than $R_n^{(1)}$ in all cases.

\item
With the simplest basis function $\q(x) = (\log x, \, x)^\T$,
$R_n^{(2)}$ has the type I error rate of $0.0625$, which
notably exceeds the nominal size of $5\%$;
the type I error rate improves significantly when the
dimension of the basis function increases (0.484, 0.474, and
0.0547 for three, four and five dimensional cases,
respectively).

\item
The powers of all four tests decrease as the dimension of
the basis function increases.
\end{enumerate}

These observations agree with our intuition. It seems that
in this particular case, choosing a three dimensional basis
function gives the best overall result: a reasonable accurate type I
error rate and also a good power.
The issue on how to choose basis function to achieve such a
balance in general is rather delicate, and we will study it
in the near future.

\begin{figure}[!ht]
\begin{center}
    \includegraphics[scale=.35]{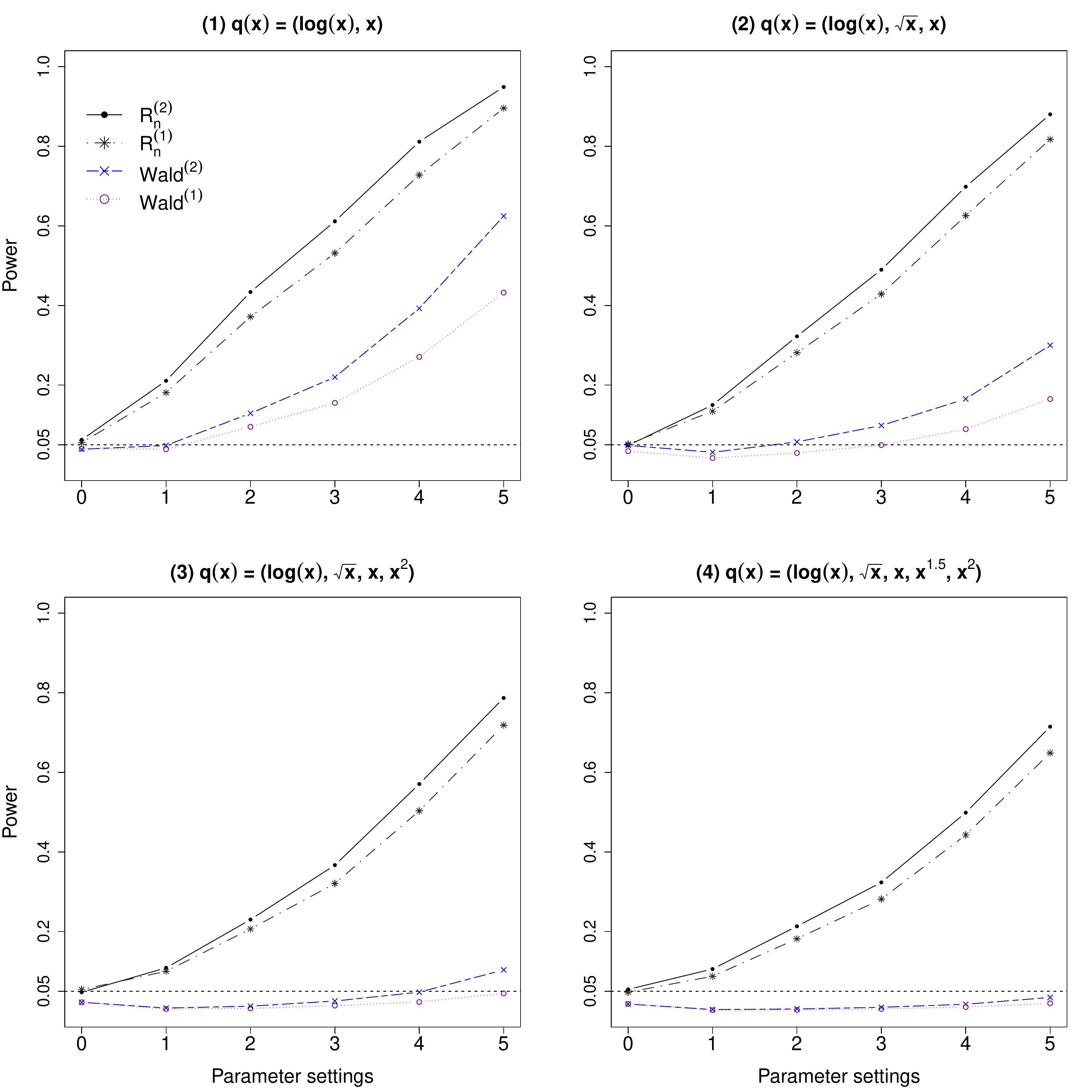}
    \caption{Power curves of $R_n^{(1)}$, $R_n^{(2)}$,
    Wald$^{(1)}$ and Wald$^{(2)}$ under DRMs with basis
    functions of different dimensions;
    Parameter setting $0$ corresponds to the null model;
    settings 1--5 correspond to alternative models.}
    \label{fig:DELRT2vs1Basis}
\end{center}
\end{figure}

\section{Analysis of lumber properties}
\label{sec:LumData}

The authors are members of the Forest Products Stochastic
Modeling Group centered at the University of British
Columbia and in that capacity are helping develop
methods for assessing the engineering strength properties of
lumber. A primary goal, one noted in Introduction, is an
effective but relatively inexpensive long term monitoring
program for the strength of lumber. One strength, which is
of primary importance, is the so--called modulus of rupture
(MOR) or ``bending strength'', which is measured in units of
$10^3$ pound--force per square inch (psi).
The Forest Products Stochastic Modeling Group collected
three MOR samples in year 2007, 2010 and 2011 with
sample sizes 98, 282 and 445, respectively. Our interest in
change over time, lead us to test the hypothesis that the
three samples come from the same lumber population.

We used basis function $\q(x) = {(\log x, \, x, \, x^2)}^\T$
for the DRM, chosen according to the characteristics of the kernel
density estimators of the MOR samples shown in Figure
\ref{fig:densPlotMOR} (a).
They seem to be well approximated by either a Gamma or a
normal distribution.
Hence, we chose a basis function that includes both $(\log x, \, x)$ 
and $(x, \, x^2)$.
To examine the adequacy of this basis function for fitting the MOR samples, we obtained 
EL kernel density estimates
based on $\{x_{kj}\}$ with weights $\{ \hat{p}_{kj} \}$
in addition to the usual kernel density estimates.
These density estimates along with
histograms of the MOR samples are shown in Figure
\ref{fig:densPlotMOR} (b) -- (d). 
We see that the EL kernel density estimates based 
on the DRM 
(the DRM fits) agree with the usual kernel density estimators 
(the Empirical fits) and the histograms well.

\begin{figure}[!ht]
\begin{center}
    \includegraphics[scale=.35]{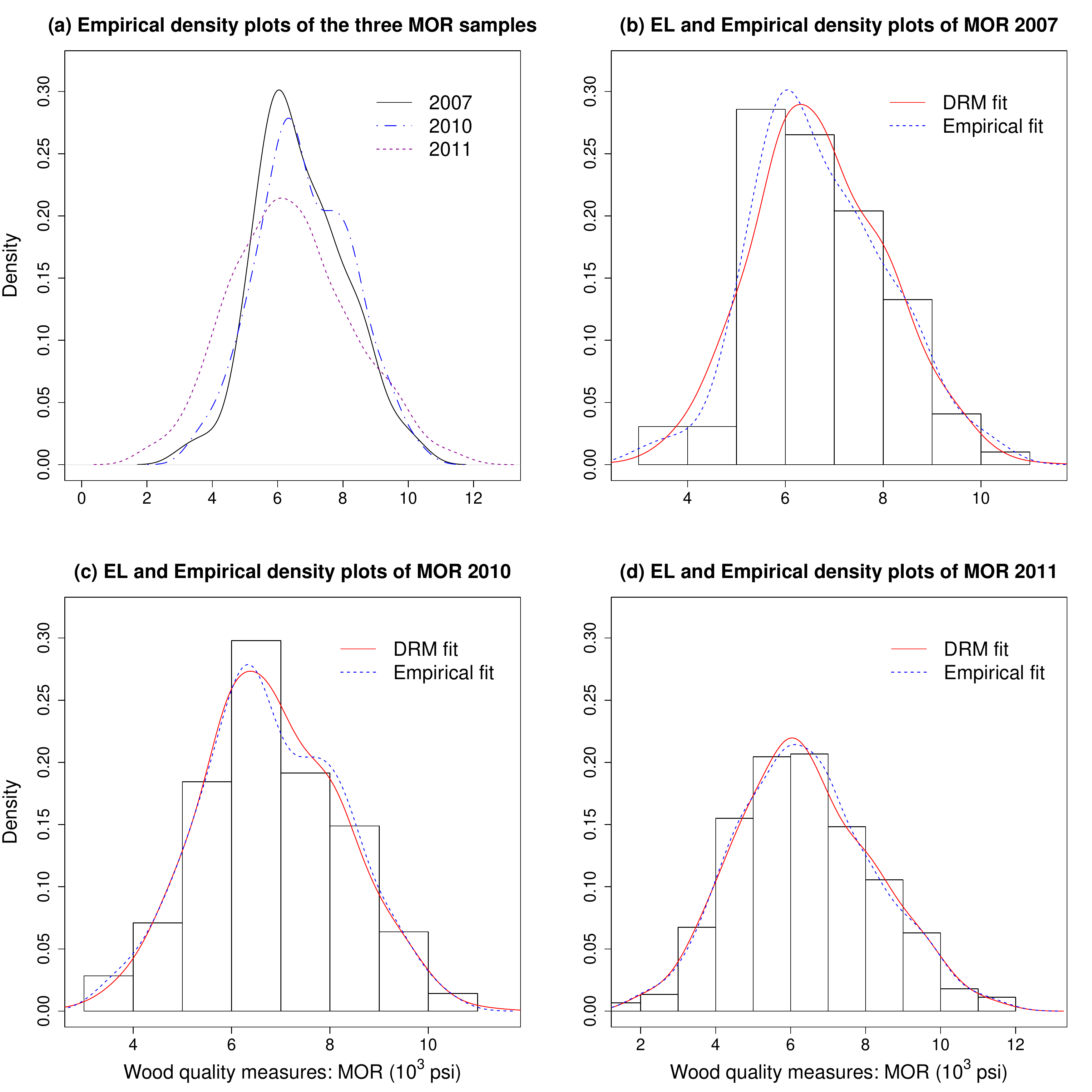}
    \caption{EL and empirical kernel density plots of the
    MOR samples}
    \label{fig:densPlotMOR}
\end{center}
\end{figure}

Let $F_{07}$, $F_{10}$ and $F_{11}$ denote the population
distributions for year 2007, 2010 and 2011, respectively. 
The p--values obtained using the DELR test, Wald test,
ANOVA and Kruskal--Wallis tests for
$H_0: F_{07} = F_{10} = F_{11}$
are respectively 3.05e-8, 2.04e-6, 2.90e-3 and 1.08e-3. 
The DRM--based tests, especially the
DELR test, have much smaller p--values.

Following the rejection of that hypothesis it is natural 
to look for its cause through pairwise comparisons. 
The p--values for pairwise comparisons are given
in Table \ref{tab:pValMORpairComp}. 
Note that the two DRM--based tests strongly suggest
$F_{11}$ is markedly different from $F_{07}$ and $F_{10}$,
while $F_{07}$ and $F_{10}$ are not significantly different.
The other two tests arrive at the same conclusion, but
without statistical significance at 5\% level. 
We also remark that the conclusion does not change 
at the 5\% level when a Bonferroni correction is applied to account 
for the multiple comparison.

In addition, if the 5\% size is strictly observed, t--test and KW test
would imply $F_{07} = F_{10}$ and $F_{07} = F_{11}$, but $F_{10} \neq F_{11}$.
This is much harder to interpret in applications.

\begin{table}
\caption{The p--values of pairwise comparisons among three MOR populations.}
\label{tab:pValMORpairComp}
\centering
\begin{tabular}{|l|cccc|}
  \hline
 & DELRT & Wald & t--test & KW \\ 
  \hline
  $\text{H}_0$: $F_{07}$=$F_{10}$ & 0.871 & 0.875 & 0.516 & 0.431 \\ 
  $\text{H}_0$: $F_{07}$=$F_{11}$ & 5.40e-4 & 7.01e-3 & 0.0579 & 0.0604 \\ 
  $\text{H}_0$: $F_{10}$=$F_{11}$ & 4.54e-8 & 1.82e-6 & 6.09e-4 & 3.95e-4 \\ 
  \hline
\end{tabular}
\end{table}

\section{Concluding remarks}

This paper has presented a new theory for testing a general class of
hypotheses under the DRM. The work
was motivated by an important application, the development of a new
long term monitoring program for the North American lumber industry.
The need for efficiency and hence small sample sizes led to our 
DRM approach where common information across samples
are pooled to gain efficiency.

The new theory is very general and flexible, making it quite
robust against misspecification of population distributions.
Our theoretical analysis and simulation confirm that the
new test has superior power than many competitors including
the likelihood ratio test based on partial likelihood under
the Cox proportional hazards model,
and does borrow strength as intended, to reduce the sample
size needed to achieve required power. The demonstration of
the use of the method on three lumber samples, shows our
method to give a more incisive assessment than competitors
through paired comparisons of the populations.

Our R package
{\emph{drmdel}}
for EL inference under DRMs, which is available on
{CRAN},
can carry out all computation tasks in this paper and those
in \citet{Chen2013}.

\section*{Appendix I: Proofs}
\label{append:proofs}
\input{proof-general.tex}
\input{proof-theorem1.tex}
\input{proof-theorem2.tex}
\input{proof-theorem3.tex}

\section*{Appendix II: Parameter values in simulation studies}
\label{append:tables}
\input{tables}

\bibliographystyle{chicago}
\bibliography{drm.bib}

\end{document}

%% file: newtheorems.tex
\theoremstyle{plain}
\newtheorem{theorem}{Theorem}
\newtheorem*{theorem*}{Theorem}

\newtheorem{lemma}{Lemma}

\theoremstyle{definition}

\theoremstyle{plain}

\newtheorem{example}{Example}
\theoremstyle{remark}

%% file: newcommands.tex
\newcommand{\E}{\text{\sc{E}}}  
\newcommand{\Var}{\text{\sc{Var}}}  
\newcommand{\Cov}{\text{\sc{Cov}}}  
\newcommand{\Prob}{\text{P}}  
\newcommand{\Deriva}{d}  

\newcommand{\diag}{\text{diag}}

\newcommand\T{{\mathpalette\raiseT\intercal}}
\newcommand\raiseT[2]{%
  \setbox0\hbox{$#1{#2}$}\raise\dp0\box0}
\newcommand{\argmax}{\text{argmax}}

\newcommand{\ind}{\mathds{1}}
\newcommand{\cl}{\mathcal{L}}
\newcommand{\h}{\boldsymbol{h}}

\newcommand{\g}{\boldsymbol{g}}
\newcommand{\altg}{\mathcal{G}}

\newcommand{\e}{\boldsymbol{e}}
\newcommand{\q}{\boldsymbol{q}}
\newcommand{\bQ}{\boldsymbol{Q}}

\newcommand{\bv}{\boldsymbol{v}}
\newcommand{\bw}{\boldsymbol{w}}

\newcommand{\bz}{\boldsymbol{z}}
\newcommand{\bc}{\boldsymbol{c}}
\newcommand{\bbeta}{\boldsymbol{\beta}}
\newcommand{\balpha}{\boldsymbol{\alpha}}
\newcommand{\btheta}{\boldsymbol{\theta}}

\newcommand{\bgamma}{\boldsymbol{\gamma}}
\newcommand{\btau}{\boldsymbol{\tau}}
\newcommand{\bxi}{\boldsymbol{\xi}}

\newcommand{\bmu}{\boldsymbol{\mu}}
\newcommand{\bsigma}{\boldsymbol{\sigma}}
\newcommand{\etab}{\boldsymbol{\eta}}
\newcommand{\bnu}{\boldsymbol{\nu}}

\newcommand{\bzeta}{\boldsymbol{\zeta}}
\newcommand{\bphi}{\boldsymbol{\phi}}

%% file: proof-general.tex
We first introduce more notations applicable to $k=0, \ldots, m$. 
Recall that
$\varphi_k(\btheta, \, x) = \exp\{ \alpha_k + \bbeta_k^\T \q(x)\}$.
We write
\[
  \cl_{n,k}(\btheta, \, x) =
  -\log \Big\{
    \sum_{r=0}^m \hat{\lambda}_r \varphi_r(\btheta, \, x)
  \Big\}
  +
  \big\{ \alpha_k + \bbeta_k^\T \q(x) \big\}
\]
with $\hat{\lambda}_r = n_r/n$ being the sample proportion.
Hence, the DEL
$
  l_n(\btheta) = \sum_{k,\, j} \cl_{n,k}(\btheta, \, x_{kj})
$
where the summation is over all possible $(k, j)$.
Let $\cl_{k}(\btheta, \, x)$ be the ``population'' version of
$\cl_{n,k}(\btheta, \, x)$ by replacing $\hat{\lambda}_r$
with its limit
$\rho_r$ in the above definition.
Let $\e_k$ be a vector of length $m$ with the $k^{th}$ entry being $1$
and the others being $0$s, and let $\delta_{ij} = 1$ when $i=j$, and $0$
otherwise. Recall the definitions (\ref{eq:def}) of
$\h(\btheta, \, x)$, $s(\btheta, \, x)$ and $H(\btheta, \, x)$.
The first order derivatives of $\cl_{k}(\btheta, \, x)$ can be written as
\begin{equation} 
\label{eq:clDeriva1}
\begin{split}
&
{\partial \cl_{k}(\btheta, \, x)}/{\partial \balpha}
  =
  (1-\delta_{k0}) \e_k - \h(\btheta, \, x)/s(\btheta, \, x), \\
&  
{\partial \cl_{k}(\btheta, \, x)}/{\partial \bbeta}
  =
\{\partial \cl_{k}(\btheta, \, x)/\partial \balpha\} \otimes \q(x).
\end{split}
\end{equation}
Similarly, we have
\begin{equation} \label{eq:clDeriva2}
  \begin{split}
  &{\partial^2 \cl_{k}(\btheta, \, x)}/{\partial \balpha \partial \balpha^\T}
  =
  -H(\btheta, \, x)/s(\btheta, \, x),
  \\
  &{\partial^2 \cl_{k}(\btheta, \, x)}/{\partial \bbeta \partial \bbeta^\T}
  =
  -\big\{
    H(\btheta, \, x)/s(\btheta, \, x)
  \big\}
  \otimes \big\{\q(x) \q^\T(x)\big\},
  \\
  &{\partial^2 \cl_{k}(\btheta, \, x)}/{\partial \balpha \partial \bbeta^\T}
  =
  -\big\{
    H(\btheta, \, x)/s(\btheta, \, x)
  \big\}
  \otimes \q^\T(x).
\end{split}
\end{equation}

The algebraic expressions of the derivatives of $\cl_{n,k}(\btheta, \,
x)$ are similar to those of $\cl_{k}(\btheta, \, x)$, only with
$\rho_r$ replaced by the sample proportion $\hat{\lambda}_r$. Note
that all entries of $\h(\btheta, \, x)$ are non--negative, and
$s(\btheta, \, x)$ exceeds the sum of all entries of
$\h(\btheta, \, x)$. Thus,
$\| \h(\btheta, \, x) / s(\btheta, \, x) \| \leq 1$
in terms of Euclidean norm, and the absolute value of each entry of
$H(\btheta, \, x)/ s(\btheta, \, x)$ is bounded by $1$.
By examining the algebraic expressions closely,
this result implies
\begin{align} \label{eq:clDerivaIneq}
  \left \vert 
  \frac{\partial^2 \cl_{n,k}(\btheta, \, x)}
    {\partial \theta_i \partial \theta_j} 
  \right \vert
  \le
 1+ \q^\T(x) \q(x)
  \ \ \  \text{and} \ \ \  
  \left \vert \frac{\partial^3 \cl_{n,k}(\btheta, \, x)}
  {\partial \theta_i \partial \theta_j \partial \theta_k} \right \vert
  \le
  \{1+ \q^\T(x) \q(x)\}^{3/2},
\end{align}
where $\theta_i$ denotes the $i^{th}$ entry of $\btheta$.

We also observed the following important relationships between the first and
second order derivatives of $\cl_{k}(\btheta, \, x)$:
\begin{align} \label{eq:keyObservation1}
  \E_0
  \left\{
    \frac{\partial \cl_{0}(\btheta^*, \, x)}{\partial \balpha}
  \right\}
  &=
  -\rho_0^{-1} U_{\balpha \balpha} \boldsymbol{1}_m,
  \quad
  \E_0
  \left\{
    \frac{\partial \cl_{0}(\btheta^*, \, x)}{\partial \bbeta}
  \right\}
  =
  -\rho_0^{-1} U_{\bbeta \balpha} \boldsymbol{1}_m,
\end{align}
and, for $k=1, \, \ldots, \, m$,
\begin{equation} \label{eq:keyObservation2}
  \begin{split}
  \E_k
  \left\{
    \frac{\partial \cl_{k}(\btheta^*, \, x)}{\partial \balpha}
  \right\}
  &=
  \rho_k^{-1} U_{\balpha \balpha} \e_k,
  \quad
  E_k \left\{
    \frac{\partial \cl_{k}(\btheta^*, \, x)}{\partial \balpha}
    \q^\T(x)
  \right\}
  =
  \rho_k^{-1} U_{\balpha \bbeta}  (\e_k \otimes I_d),
  \\
  \E_k
  \left\{
    \frac{\partial \cl_{k}(\btheta^*, \, x)}{\partial \bbeta}
  \right\}
  &=
  \rho_k^{-1} U_{\bbeta \balpha} \e_k,
  \quad
  E_k \left\{
    \frac{\partial \cl_{k}(\btheta^*, \, x)}{\partial \bbeta}
    \q^\T(x)
  \right\}
  =
  \rho_k^{-1} U_{\bbeta \bbeta} (\e_k \otimes I_d).
\end{split}
\end{equation}

The assumption that
$
  \int \exp\{ \bbeta_k^\T \q(x) \} d F_0 < \infty
$
for $\btheta$ in a neighbourhood of $\btheta^*$ implies that the moment
generating function of $\q(x)$ with respect to each $F_k$,
exists in a neighbourhood of $\boldsymbol{0}$. Hence, all finite order moments
of $\q(x)$ with respect to each $F_k$ are finite. This fact and inequalities
(\ref{eq:clDerivaIneq}) reveal that the second and third
order derivatives of
$l_n(\btheta)$ are bounded by an integrable function.

Under the assumption of Theorem \ref{thm:DELRT} that 
$\int \bQ(x)\bQ^\T(x) dF_0$ is positive definite, the information
matrix $U$ given by (\ref{eq:Uexpression}) is positive definite.
As a reminder, $\bQ(x) = {(1, \q^\T(x))}^\T$.

%% file: proof-theorem1.tex
\subsection*{A.1. Proof of Theorem \ref{thm:DELRT}}

Under the null hypothesis (\ref{hp:composite}),
we show that the DELR statistic is approximated by a
quadratic form that has a chi--square limiting distribution.
We first give two key lemmas.

Let $T  =
  {\rho_0}^{-1}  \boldsymbol{1}_m \boldsymbol{1}_m^\T
  +
  \diag \{ \rho_1^{-1}, \ldots, \rho_m^{-1} \}$
and $W = \diag\{ T, \, 0_{md \times md}\}$.
Put $\bv = n^{-1/2} {\partial l_n(\btheta^*)} / {\partial \btheta}$.
Let $\E(\cdot)$ be the usual expectation operator and $\E_k(\cdot)$ be
the expectation operator respect $F_k$.

\begin{lemma}[Asymptotic properties of the score function]
\label{lemma:propScore}
Under the conditions of Theorem \ref{thm:DELRT},
$\E \bv = \boldsymbol{0}$ and $\bv$ is asymptotically multivariate
normal with mean $\boldsymbol{0}$ and covariance matrix
$
  V  =  U - UWU
$.
\end{lemma}

\begin{proof}[\sc{Proof}]
Denote $\bmu_k = \E_k \{ \partial \cl_{n,k}(\btheta^*, \, x)/\partial \btheta\}$.
We can verify that
\mbox{$\E \bv = n^{1/2} \sum_{k=0}^m \hat{\lambda}_k \bmu_k =
\boldsymbol{0}$}.
Hence, we have
\begin{align*} 
  \bv
  &=
  \sum_{k=0}^m  \hat{\lambda}_k^{1/2}
  \big\{
    n_k^{-1/2} \sum_{j=1}^{n_k}
    \big( 
      {\partial \cl_{n,k}(\btheta^*, \, x_{kj})}/{\partial \btheta}
      - \bmu_k
    \big)
  \big\}.
\end{align*}
Clearly, each term in curly brackets
is a centered sum of iid random variables with finite covariance matrices.
Thus, they are all asymptotically normal with appropriate covariance matrices. 
In addition, these terms are independent of each other, $\hat{\lambda}_k = n_k/n$
are non--random with a limit $\rho_k$. Therefore, the linear combination
is also asymptotically normal.

What left is to verify the form of the asymptotic covariance
matrix. The asymptotic covariance matrix of each term 
in curly brackets is given by
\begin{align*}
  V_k
  =
  \E_k
  \left\{
    ({\partial \cl_{k}(\btheta^*, \, x)}/{\partial \btheta})
    ({\partial \cl_{k}(\btheta^*, \, x)}/{\partial \btheta^\T})
  \right\}
  -
   \bmu_k \bmu_k^\T,
\end{align*}
and hence the overall asymptotic variance matrix is
$V  =  \sum_{k=0}^m \rho_k V_k$.
In addition, it is easy to verify that 
\[
  \sum_{k=0}^m \rho_k
  \E_k
  \left\{
    ({\partial \cl_{k}(\btheta^*, \, x)}/{\partial \btheta})
    ({\partial \cl_{k}(\btheta^*, \, x)}/{\partial \btheta^\T})
  \right\}
= U
\]
and we also find
$ \sum_{k=0}^m \rho_k \bmu_k \bmu_k^\T  = UWU$ by
(\ref{eq:keyObservation1}) and (\ref{eq:keyObservation2}).
Thus, $V = U - UWU$ and this completes the proof.
\end{proof}

\begin{lemma}[Quadratic form decomposition formula]
  \label{lemma:matDecomp}
  Let $\bz^\T = (\bz_1^\T, \, \bz_2^\T)$ be a vector of
  length $m+n$, partitioned in agreement with $m$ and $n$,
  and $\Sigma$ be a $(m+n) \times (m+n)$ a nonsingular matrix with partition
  \begin{align*}
    \Sigma
    =
    \begin{pmatrix}
      \underset{m \times m}{A} & \underset{m \times n}B \\
      \underset{n \times m}{B^\T} & \underset{n \times n}{C}
    \end{pmatrix}.
  \end{align*}
  When $A$ is nonsingular, so is $C - B^\T A^{-1}B$ and 
  \begin{align*}
    \bz^\T \Sigma^{-1} \bz =
    \big(\bz_2 - B^\T A^{-1} \bz_1 \big)^\T
    \big(C - B^\T A^{-1}B\big)^{-1}
    \big(\bz_2 - B^\T A^{-1} \bz_1 \big) + \bz_1^\T A^{-1} \bz_1.
  \end{align*}
\end{lemma}

One can verify the above conclusion directly or
refer to Theorem 8.5.11 of \citealt{Harville2008}.

\begin{proof}[\sc{Proof of Theorem \ref{thm:DELRT}}]

We first work on quadratic expansions of
$l_n(\hat \btheta)$ and $l_n(\tilde \btheta)$ under the null model. 
The difference of the two quadratic forms is then shown to
have a chi--square limiting distribution.

Recall $\bv = n^{-1/2} {\partial l_n(\btheta^*)}/{\partial \btheta}$.
By expanding $l_n(\btheta)$ at $\btheta^*$, we get
\begin{align*}
  l_n(\btheta)
  =
  l_n(\btheta^*)
  + \sqrt{n} \bv^\T (\btheta - \btheta^*)
  - (1/2) n (\btheta - \btheta^*)^\T U_n (\btheta - \btheta^*)
  + \epsilon_n
\end{align*}
where $\epsilon_n = O_p(n^{-1/2})$ when $\btheta - \btheta^* = O_p(n^{-1/2})$
because the third derivative is bounded by an integrable function
shown in (\ref{eq:clDerivaIneq}).
Ignoring $\epsilon_n$, the leading term in this expansion is maximized
when 
\begin{align*}
\btheta - \btheta^* 
= 
n^{-1/2} U_n^{-1} \bv + o_p(n^{-1/2}).
\end{align*}
At the same time, the DEL $l_n(\btheta)$ is by definition
maximized at $\btheta = \hat \btheta$, and $\hat \btheta$ is
known to be root--$n$ consistent (\citealt{Chen2013} and
\citealt{Zhang2002}), hence
\begin{align*}
\hat \btheta -\btheta^* 
= 
n^{-1/2} U_n^{-1} \bv + o_p(n^{-1/2})
=
n^{-1/2} U^{-1} \bv + o_p(n^{-1/2}),
\end{align*}
which leads to
\begin{align} 
\label{eq:lnExpansionOriginal}
  l_n(\hat \btheta)
  =
  l_n(\btheta^*) + (1/2) \bv^\T U^{-1} \bv + o_p(1).
\end{align}


Next, we work on an expansion for $l_n(\tilde \btheta)$ under the null
model. Recall that $\bbeta$ is part of $\btheta$.
We express the null hypothesis $\g(\bbeta) = {\bf 0}$ in
another equivalent form.
Let $\bbeta^*$ represent a null model.
Recall that
$\g: \mathbb{R}^{md} \to \mathbb{R}^{q}$
is thrice differentiable in a neighbourhood of
$\bbeta^*$ with full rank Jacobian matrix
\mbox{$
\bigtriangledown = \partial \g(\bbeta^*)/\partial \bbeta
$}.
When $q < md$, by the implicit function theorem
\citep[8.5.4, Theorem 1]{Zorich2004},
there exists a unique function $\altg$:
$\mathbb{R}^{md-q} \to \mathbb{R}^{md}$, such that 
$\g(\bbeta) = 0$ if and only if $\bbeta = \altg(\bgamma)$
for some $\bbeta$ and $\bgamma$ in a corresponding
neighbourhoods of $\bbeta^*$ and $\bgamma^*$
respectively.
In addition, $\altg$ is also thrice differentiable in a
neighbourhood of $\bgamma^*$, and its Jacobian is
\begin{align*}
J = \partial \altg(\bgamma^*)/\partial \bgamma
  = (
    -(\bigtriangledown_1^{-1} \bigtriangledown_2)^\T, \,  I_{md-q}
  )^\T.
\end{align*}
This Jacobian is the same as the matrix $J$ in Theorem
\ref{thm:localPower}.
When $q=md$, by the inverse function theorem \citep[8.6.1,
Theorem 1]{Zorich2004}, $\g$ is invertible at $\bbeta^*$,
i.e. $\bbeta^* = \g^{-1}(\bf 0)$. Hence, in this case, $\g$
defines a simple hypothesis testing problem with $\bbeta$
being fully specified to be $\g^{-1}(\bf 0)$ in the null.

We first look at the case of $q < md$.
With the above representaion of the null model, the
DRM parameter under the null hypothesis
is $\btheta = (\balpha, \, \altg(\bgamma))$.
Hence, we may write the likelihood function under null model as
\begin{align*}
\ell_n(\balpha, \, \bgamma) =  l_n(\balpha, \, \altg(\bgamma)).
\end{align*}
Let $(\tilde \balpha, \, \tilde \bgamma)$ be the maximal point of
$\ell_n(\balpha, \bgamma)$.
Clearly, $\ell_n(\balpha, \, \bgamma)$ has the same properties
as $l_n(\btheta)$ and $\ell_n(\tilde \balpha, \, \tilde \bgamma)$
has a similar expansion as
\eqref{eq:lnExpansionOriginal}.
Partition $\bv$ into
$\bv_1 = n^{-1/2} \partial l_n(\btheta^*)/\partial \balpha$ and
$\bv_2 = n^{-1/2} \partial l_n(\btheta^*)/\partial \bbeta$.
Note that
\begin{align*}
n^{-1/2} \partial \ell_n(\balpha^*, \bgamma^*)/\partial \balpha
= 
n^{-1/2} \partial l_n(\btheta^*)/\partial \balpha = \bv_1.
\end{align*}
By the chain rule, 
\begin{align} \label{eq:scoreNullPartition}
n^{-1/2} \partial \ell_n(\balpha^*, \bgamma^*)/\partial  \bgamma 
  =
n^{-1/2} J^\T \{\partial l_n(\btheta^*)/\partial \bbeta) \}
  =
J^\T \bv_2.
\end{align}
Similarly, the new information matrix is found to be
\begin{align*}
  \tilde{U}
  =
      \begin{pmatrix}
   I_m   &   0   \\
   0   &  J^\T
  \end{pmatrix} 
    \begin{pmatrix}
    U_{\balpha\balpha}   &    U_{\balpha\bbeta}    \\
   U_{\bbeta\balpha}   &   U_{\bbeta\bbeta}
  \end{pmatrix}
   \begin{pmatrix}
   I_m   &   0   \\
   0   &  J
  \end{pmatrix} 
  =
  \begin{pmatrix}
    {U_{\balpha\balpha}}    &    {U_{\balpha\bbeta}J}    \\
    {J^\T U_{\bbeta\balpha}}    &    {J^\T U_{\bbeta\bbeta} J}
  \end{pmatrix}.
\end{align*}
Consequently, we have
\begin{align*} 
  l_n(\tilde \btheta)
  =
\ell_n(\tilde \balpha, \tilde \bgamma)
  =
\ell_n(\balpha^*, \bgamma^*)
+
 (1/2)
(\bv_1^\T, \bv_2^\T J)
  \tilde{U}^{-1}
(\bv_1^\T, \bv_2^\T J)^\T
  +
  o_p(1).
\end{align*}

Combining (\ref{eq:lnExpansionOriginal}) and the above expansion, 
and noticing that $\ell_n(\balpha^*, \bgamma^*) = l_n(\btheta^*)$,
we have
\begin{align*} 
  R_n
  =
  2\{ l_n(\hat \btheta) - l_n(\tilde \btheta) \}
  =
  \bv^\T U^{-1} \bv
  -
  (\bv_1^\T, \bv_2^\T J)
  \tilde U^{-1}
  (\bv_1^\T, \bv_2^\T J)^\T
  + o_p(1).
\end{align*}
Applying Lemma \ref{lemma:matDecomp} to the two quadratic forms on
the right hand side (RHS) of the above expansion, we get
\begin{equation} \label{eq:quadDecomp}
\begin{split}
  &\bv^\T U^{-1} \bv
  = \bxi^\T \Lambda^{-1} \bxi +
  \bv_1^\T U_{\balpha\balpha}^{-1} \bv_1,
  \\
  &(\bv_1^\T, \bv_2^\T J)
  \tilde U^{-1}
  (\bv_1^\T, \bv_2^\T J)^\T
  =
  \bxi^\T J {( J^\T \Lambda J )}^{-1} J^\T \bxi
  +
  \bv_1^\T U_{\balpha\balpha}^{-1} \bv_1,
\end{split}
\end{equation}
where
$
\bxi =
(- U_{\bbeta\balpha} U_{\balpha\balpha}^{-1}, \  I_{md}) \bv
$ and
$\Lambda =
U_{\bbeta \bbeta} -  U_{\bbeta \balpha} U_{\balpha \balpha}^{-1}
U_{\balpha \bbeta}$ is defined in Theorem \ref{thm:localPower}.
We then obtain the following expansion
\begin{align} \label{eq:RnExpansion}
  R_n
  =
  2\{ l_n(\hat \btheta) - l_n(\tilde \btheta) \}
  =
  \bxi^\T
  \{ \Lambda^{-1} - J {( J^\T \Lambda J )}^{-1} J^\T \}
  \bxi
  + o_p(1).
\end{align}
Recall that, by Lemma \ref{lemma:propScore},
$\bv$ is asymptotically $N(\boldsymbol{0}, \, U-UWU)$,
so $\bxi$ is asymptotic normal with mean $\bf 0$ and covariance matrix
$
  (- U_{\bbeta\balpha} U_{\balpha\balpha}^{-1}, \  I_{md})  (U - UWU)
  (- U_{\bbeta\balpha} U_{\balpha\balpha}^{-1}, \  I_{md})^\T = \Lambda$,
where the last equality is obtained using the expression of $W$ given in
Lemma \ref{lemma:propScore}.

The last step is to verify the quadratic form in the above expansion of $R_n$
has the claimed limiting distribution. We can easily check that
\begin{align*}
  \Lambda^{1/2}
  \{ \Lambda^{-1} - J {( J^\T \Lambda J )}^{-1} J^\T  \}
  \Lambda^{1/2}
\end{align*}
is idempotent.
Moreover, the trace of the above idempotent matrix is found to be $q$.
Therefore, by Theorem 5.1.1 of \citet{Mathai1992}, the quadratic form
in expansion (\ref{eq:RnExpansion}), and hence also $R_n$, has a
$\chi_q^2$ limiting distribution.

The above proof is applicable to  $q < md$. 
When $q = md$, the value of $\bbeta$ is fully specified. Hence,
the maximization under null is solely with respect to $\balpha$
and we easily find
\begin{align*}
  l_n(\tilde \btheta)
  =
  l_n(\btheta^*) + 
  (1/2)  \bv_1^\T U_{\balpha\balpha}^{-1} \bv_1
  +
  o_p(1).
\end{align*}
This, along with the expansion (\ref{eq:lnExpansionOriginal})
of $l_n(\hat \btheta)$ and expression (\ref{eq:quadDecomp}),
implies that
$
  R_n = 
  \bxi^\T
  \Lambda^{-1}
  \bxi
  + o_p(1)
$.
Just as the proof for the case of $q < md$, the limiting
distribution of the above $R_n$ is seen to be $\chi_{md}^2$.
\end{proof}

%% file: proof-theorem2.tex
\subsection*{A.2. Proof of Theorem \ref{thm:localPower}}

We first sketch out the proof of Theorem \ref{thm:localPower}. 
Let $\bbeta^*$ be a specific parameter value under the null hypothesis 
and $\{F_k\}$ be the corresponding distribution functions.
Let $\{G_k\}$ be the set of distribution functions satisfying the DRM 
with parameter given by
$
  \bbeta_k
  =
  \bbeta_k^* + n_k^{-1/2} \bc_k
$,
$k = 1, \, \ldots, \, m$,
and $G_0=F_0$.
When the samples are generated from the $\{G_k\}$,
we still have that the DELR statistic is approximated by the quadratic
form on the RHS of (\ref{eq:RnExpansion}).
The limiting distribution of $R_n$ is therefore determined by that of 
$\bv = n^{-1/2} {\partial l_n(\btheta^*)} / {\partial \btheta}$.
According to Le Cam's third lemma \citetext{\citealt{vanderVaart2000}, 6.7},
$\bv$ has a specific limiting distribution under the
$\{G_k\}$ 
if $\bv$ and
$\sum_{k,j} \log \{{dG_k(x_{kj})}/{dF_k(x_{kj})}\}$,
under the $\{F_k\}$, 
are jointly normal with a particular mean and variance structure.
The core of the proof then is to establish that structure.

For each $k=0, \, \ldots, \, m$,
let $\Var_k(\cdot)$ and $\Cov_k(\cdot, \, \cdot)$ be the
variance and covariance operators with respect to $F_k$,
respectively.

\begin{lemma} \label{lemma:scoreAlt}
Under the conditions of Theorem \ref{thm:DELRT}
and the distribution functions $\{G_k\}$,
$\bv$ is asymptotically normal with mean
$
  \btau
  =
  \sum_{k=1}^m \sqrt{\rho_k}
  \Cov_k \{
    {\partial \cl_{k}(\btheta^*, \, x)}/{\partial \btheta}, \q^\T(x) \}
  \bc_k
$
and covariance matrix $V=U-UWU$ as given in Lemma
\ref{lemma:propScore}.
\end{lemma}

\begin{proof}[\sc{Proof of Lemma \ref{lemma:scoreAlt}}]

We first expand
$
  \bw_k =
  \sum_{j=1}^{n_k} \log \{{\Deriva G_k(x_{kj})}/{\Deriva F_k(x_{kj})}\}
$.
Notice that
\begin{align*} 
dG_k(x)/dF_k(x) 
&= 
\exp\{\alpha_k + \bbeta_k \q(x)\}/\exp\{\alpha_k^* + \bbeta_k^* \q(x)\}\\
&= 
\exp\{ \alpha_k -\alpha_k^* + n_k^{-1/2} \bc_k \q(x) \}.
\end{align*}
Because $\alpha_k$ and $\alpha_k^*$ are normalization constants,
we have
\[
\exp\{ \alpha_k^* -\alpha_k\}
=
\int 
\exp\{\alpha_k^* + ({\bbeta_k^*}^\T + n_k^{-1/2} \bc_k^\T )\q(x)\}  dF_0(x).
\]
Ignoring terms of order $n^{-3/2}$ and higher, it leads to
\begin{align*}
\exp\{ \alpha_k^* -\alpha_k\}
& =
\int 
\exp \{  n_k^{-1/2} \bc_k^\T \q(x) \}
\exp\big\{\balpha_k^* + {\bbeta_k^*}^\T \q(x) \big\} d F_0(x)\\
&\approx
\int 
\big\{ 
1 +   n_k^{-1/2} \bc_k^\T \q(x)  + (2n_k)^{-1} (\bc_k^\T \q(x))^2
\big\} d F_k(x).
\end{align*}
Denote $\bnu_k = \E_k \q(x)$.
Then, it is further simplified to
\[
\exp\{ \alpha_k^* -\alpha_k\}
\approx
1 + n_k^{-1/2} \bc_k^\T \bnu_k +
(2n_k)^{-1} \bc_k^\T \E_k(\q^2(x)) \bc_k.
\]
Hence, ignoring a $O(n^{-3/2})$ term, we have
\begin{align*}
\log \{dG_k(x)/dF_k(x)\}
&\approx
n_k^{-1/2} \bc_k \q(x) 
- \log \{
1 + n_k^{-1/2} \bc_k^\T \bnu_k +
(2n_k)^{-1} \bc_k^\T \E_k(\q^2(x)) \bc_k
\}.
\end{align*}
Write $\bsigma_k = \Var_k(\q(x))$.
Expanding the logarithmic term on the RHS, we get
\begin{align*}
&\log \{ 1 + n_k^{-1/2} \bc_k^\T \bnu_k +
(2n_k)^{-1} \bc_k^\T \E_k(\q^2(x)) \bc_k\}
\\
=&
n_k^{-1/2} \bc_k^\T \bnu_k +
(2n_k)^{-1} \bc_k^\T \E_k(\q^2(x)) \bc_k
-
n_k \bc_k^\T \{\bnu_k \bnu_k^\T\} \bc_k
+ O(n^{-3/2})
\\
=&
n_k^{-1/2} \bc_k^\T \bnu_k
+(2 n_k)^{-1} \bc_k^\T \bsigma_k \bc_k
+ O(n^{-3/2}).
\end{align*}
Therefore 
\begin{align*}
\log \{dG_k(x)/dF_k(x)\}
&=
n_k^{-1/2} \bc_k^\T \{ \q(x) - \bnu_k\}
-(2 n_k)^{-1} \bc_k^\T \bsigma_k \bc_k
+ O(n^{-3/2}).
\end{align*}
Summing over $j$,  we get, for each $k$,
\begin{align*}
  \bw_k=
\sum_{j=1}^{n_k} \log \{{\Deriva G_k(x_{kj})}/{\Deriva F_k(x_{kj})}\}
&=
 n_k^{-1/2} \bc_k^\T  \sum_{j = 1}^{n_k} \{ \q(x_{kj}) - \bnu_k\}
 - 
 (1/2)  \bc_k^\T \bsigma_k \bc_k
    + O(n^{-1/2}).
\end{align*}
When $k=0$, we have $\bc_0 = 0$.

Recall that
$
l_n(\btheta^*) = \sum_{k,j} \cl_{n,k}(\btheta^*, \, x_{kj})
$ and $\hat \lambda_k = n_k/n$ whose limit is $\rho_k$,
we have
\begin{align*}
  \left( 
  \begin{array}{c}
    \bv
    \\
    \sum_{k} \bw_k
  \end{array} 
  \right )
&\approx
  \sum_{k=0}^{m} \frac{1}{\sqrt{n_k}} \sum_{j = 1}^{n_k}
  \left( \begin{array}{c}
    \sqrt{\rho_k}
    \left\{
      {\partial \cl_{n,k}(\btheta^*, \, x_{kj})}/{\partial \btheta}
      -
      \bmu_k
    \right\}
    \\
    \bc_k^\T \{ \q(x_{kj}) - \bnu_k\}
  \end{array} \right)
  -
  \sum_{k=0}^{m}
  \left( \begin{array}{c}
    \boldsymbol{0}
    \\
    \frac{1}{2} \bc_k^\T \bsigma_k \bc_k
  \end{array} \right),
\end{align*}
which is seen to be jointly asymptotically normal
under the null distributions $\{F_k\}$. The corresponding mean
vector and variance matrix are given by
\[
  \left( \boldsymbol{0}^\T, \ 
    - \frac{1}{2} \sum_k \bc_k^\T \bsigma_k \bc_k
  \right)^\T
  \ \text{and} \ 
  \left( \begin{array}{cc}
    V &   \btau    \\
    \btau^\T    &  \sum_k \bc_k^\T \bsigma_k \bc_k
  \end{array} \right),
\]
where $ \btau$ is the one given in the Lemma.
Because the second entry of the mean vector
equals negative half of the lower--right entry of the covariance matrix,
the condition of Le Cam's third lemma is satisfied.
By that lemma, we conclude that $\bv$ has a normal limiting
distribution with mean $\btau$ and covariance matrix $V$
under the local alternative distributions $\{G_k\}$.
\end{proof}

\begin{proof}[\sc{Proof of Theorem \ref{thm:localPower}}]
  We first show that, under the $\{G_k\}$, the DELR statistic $R_n$ is still
  approximated by the quadratic form on the RHS of
  (\ref{eq:RnExpansion}).

  Under the $\{G_k\}$, we still have
  $
  -n^{-1}  {\partial^2 l_n(\btheta^*)} /\partial \btheta \partial
  \btheta^\T  \to U$
  and
  $
    \bv
    = O_p(1)
  $.
  In addition, $\hat \btheta$ still admits the expansion
  \begin{align*}
    \sqrt{n}(\hat \btheta - \btheta^*)
    =
    U^{-1} \bv + o_p(1)
    =
    O_p(1),
  \end{align*}
  and hence it is root--$n$ consistent for $\btheta^*$.
  Similarly, the constrained MELE $\tilde \btheta$ is also root--$n$
  consistent for $\btheta^*$ under the $\{G_k\}$.
  The root--$n$ consistency of $\hat \btheta$ and
  $\tilde \btheta$ imply
  \begin{align*}
    R_n
    = 
    \bxi^\T
    \{ \Lambda^{-1} - J {( J^\T \Lambda J )}^{-1} J^\T \}
    \bxi
    + o_p(1)
  \end{align*}
  when $q < md$, and
  $R_n = \bxi^\T \Lambda^{-1} \bxi + o_p(1)$ when $q=md$.
  The matrix in the quadratic form of the expansion of $R_n$ is the
  same as that in (\ref{eq:RnExpansion}).
  What has changed is the distribution of
  $
    \bxi
    =
    (- U_{\bbeta\balpha} U_{\balpha\balpha}^{-1}, \  I_{md})
    \bv
  $.

  By Lemma \ref{lemma:scoreAlt}, under the local alternative
  $\{G_k\}$, $\bv$ is asymptotically $N(\btau, \, V)$. Hence $\bxi$ also
  has a normal limiting distribution. Since the
  asymptotic covariance matrix of $\bv$ is the same as that
  under the $\{F_k\}$,
  the asymptotic covariance matrix of $\bxi$ is still $\Lambda$ as we
  have shown in the proof of Theorem \ref{thm:DELRT}.
  The mean of the limiting distribution of $\bxi$ now is
  $\bmu = (- U_{\bbeta\balpha} U_{\balpha\balpha}^{-1}, \  I_{md})
  \btau = \Lambda \etab$,
  where $\etab$ is defined in Theorem \ref{thm:localPower} and the
  last equality is derived using (\ref{eq:keyObservation2}).

  In the proof of Theorem \ref{thm:DELRT}, we have verified that the
  matrix
  \begin{align*}
    A =
  \Lambda^{1/2}
    \{ \Lambda^{-1} - J {( J^\T \Lambda J )}^{-1} J^\T \}
  \Lambda^{1/2}
  \end{align*}
  is idempotent with rank $q$.
  Hence, by Corollary 5.1.3a of \citet{Mathai1992}, the quadratic form
  in the above expansion of $R_n$, and hence $R_n$, has the claimed 
  non--central chi--square limiting distribution.

  In the last step we verify the condition for positiveness of the
  non--central parameter $\delta^2$.
  When $q=md$, $\delta^2  = \etab^\T \Lambda \etab > 0$ because
  $\Lambda$ is positive definite. When $q<md$,
  $
    \delta^2
    =
    (\etab^\T \Lambda^{1/2})
    A
    (\Lambda^{1/2} \etab)
  $.
  We verified that $A$ is an idempotent matrix. Hence, $A$ is positive
  semidefinite and $\delta^2 \ge 0$. Moreover, $\delta^2 = 0$ if and only if
  $\Lambda^{1/2} \etab$ is in the null space of $A$.
  The null space of $A$ is the column space of
  $
    I -A =\Lambda^{1/2} J \big(J^\T \Lambda J\big)^{-1} J^\T \Lambda^{1/2}
  $,
  which is just the column space of $\Lambda^{1/2}J$.
  It is easily verified that $\Lambda^{1/2} \etab$ is in the column
  space of $\Lambda^{1/2}J$ if and only if $\etab$ is in the column
  space of $J$. Hence $\Lambda^{1/2} \etab$ is in the null space of
  $A$ and $\delta^2 = 0$ if and only if $\etab$ is in the column space
  of $J$.
\end{proof}

%% file: proof-theorem3.tex
\subsection*{A.3. Proof of Theorem \ref{thm:pwComp}}

We first introduce a useful notation for Schur complements
that will be frequently used in the subsequent proofs.
Let matrix
\begin{align*}
  M
  =
  \begin{pmatrix}
    {A}
    &
    {B}
    \\
    {C}
    &
    {D}
  \end{pmatrix}
\end{align*}
be nonsingular.
We write $M/A = D - CA^{-1}B$ and call it the \emph{Schur complement of $M$
with respect to its upper--left block $A$}.
Also, we write $M/D = A - BD^{-1}C$ and call it the \emph{Schur
complement of $M$ with respect to its lower--right block $D$}.

Recall that we defined two DELRT statistics $R_n^{(1)}$ and
$R_n^{(2)}$ which are constructed using the samples from
only the first $r+1$ populations $F_0, \, \cdots, \, F_r$,
and the samples from all the populations, respectively.
Let $U$ be the information matrix based on all $m+1$ samples
($R_n^{(2)}$), and $\tilde{U}$ be that based on the first $r+1$
samples ($R_n^{(1)}$). Similar to the partition of $U$, we partition
$\tilde{U}$ to
$\tilde{U}_{\balpha\balpha}$, $\tilde{U}_{\balpha\bbeta}$,
$\tilde{U}_{\bbeta\balpha}$ and $\tilde{U}_{\bbeta\bbeta}$,
and similar to the definition $\Lambda = U/U_{\balpha\balpha}$
given in Theorem \ref{thm:localPower}, we define
$\tilde{\Lambda} = \tilde{U}/\tilde{U}_{\balpha\balpha}$.
We also partition $\Lambda$ as
\begin{align*}
  \Lambda
  =
  \begin{pmatrix}
    \Lambda_a & \Lambda_b
    \\
    \Lambda_b^T & \Lambda_c
  \end{pmatrix},
\end{align*}
where $\Lambda_a$ is the upper--left $rd \times rd$ block of
$\Lambda$.

The null hypothesis of (\ref{hp:comparison}) under
investigation contains a constraint $\g(\bzeta)=\bf 0$ with
$\bzeta^\T = ( \bbeta_1^\T, \, \ldots, \, \bbeta_r^\T )$
related only to populations $F_0, \cdots, F_r$. As noted in
the proof of Theorem \ref{thm:DELRT}, this null constraint
is equivalent to $\bzeta = \altg(\bgamma)$ for some smooth
function $\altg$: $\mathbb{R}^{rd-q} \to \mathbb{R}^{rd}$
and parameter vector $\bgamma$. Denote the Jacobian of
$\altg$ evaluated at $\bgamma^*$ as $J$.
By Theorem \ref{thm:localPower}, under the $\{G_k\}$ defined
by the local alternative model (\ref{eq:localAltCompare}),
$R_n^{(1)}$ and $R_n^{(2)}$ both have non--central
chi--square limiting distributions of $q$ degrees of
freedom, but with different non--central parameters
$\delta_1^2$ and $\delta_2^2$, respectively.
We also know that for $R_n^{(1)}$,
\begin{align*} 
  \delta^2_1
  &=
  \rho \tilde{\etab}^\T
  \left\{
    \tilde{\Lambda}
    -
    \tilde{\Lambda} J (J^\T \tilde{\Lambda}J)^{-1} J^\T \tilde{\Lambda}
  \right\}
  \tilde{\etab},
\end{align*}
where
$\tilde{\etab} = (\rho_1^{-1/2} \bc_1^\T, \, \ldots, \rho_r^{-1/2}
\bc_r^\T )$.
Moreover, under the same local alternative model, for
$R_n^{(2)}$, we have
$\etab^\T = (\tilde{\etab}^\T, \, 0_{m-r}^\T)$
and the corresponding Jacobian matrix of the null mapping is
$J_2 = \diag(J, \, I_{(m-r)d})$.
Thus
\begin{align*} 
  \delta^2_2
  &=
  \etab^\T
  \left\{
    \Lambda
    -
    \Lambda J_2
    (J_2^\T \Lambda J_2)^{-1}
    J_2^\T \Lambda
  \right\}
  \etab.
\end{align*} 
Let $A$ denote the upper--left $rd \times rd$ block of 
$
  \Lambda
  -
  \Lambda J_2
  (J_2^\T \Lambda J_2)^{-1}
  J_2^\T \Lambda
$.
Since $\etab$ consists of $\tilde{\etab}$ and a zero vector,
we have
\begin{align*}
  \delta^2_2
  =
  \tilde{\etab}^\T
  A
  \tilde{\etab},
\end{align*}
The upper--left block of $\Lambda$ is $\Lambda_a$.
By the quadratic form decomposition formula of Lemma
\ref{lemma:matDecomp}, the upper--left block of
$
  \Lambda J_2
  (J_2^\T \Lambda J_2)^{-1}
  J_2^\T \Lambda
$
is found to be
\begin{align*}
  (\Lambda/\Lambda_c) J
  (J^\T (\Lambda/\Lambda_c) J)^{-1}
  J^\T (\Lambda/\Lambda_c)
  +
  \Lambda_b \Lambda_c^{-1} \lambda_b^T.
\end{align*}
Hence, the expression of $\delta^2_2$ becomes
\begin{align*} 
  \delta^2_2
  &=
  \tilde{\etab}^\T
  A
  \tilde{\etab}
  \\
  &=
  \tilde{\etab}^\T
  \left\{
    \Lambda_a - \Lambda_b \Lambda_c^{-1} \lambda_b^T
    -
    (\Lambda/\Lambda_c) J
    (J^\T (\Lambda/\Lambda_c) J)^{-1}
    J^\T (\Lambda/\Lambda_c)
  \right\}
  \tilde{\etab}
  \\
  &=
  \tilde{\etab}^\T
  \left\{
    (\Lambda/\Lambda_c)
    -
    (\Lambda/\Lambda_c) J
    (J^\T (\Lambda/\Lambda_c) J)^{-1}
    J^\T (\Lambda/\Lambda_c)
  \right\}
  \tilde{\etab}.
\end{align*}
Therefore, to show the claimed result $\delta_2^2 \ge \delta_1^2$, it
suffices to show that
\begin{align} \label{eq:pwCompToShow}
  (\Lambda/\Lambda_c)
  -
  (\Lambda/\Lambda_c) J
  (J^\T (\Lambda/\Lambda_c) J)^{-1}
  J^\T (\Lambda/\Lambda_c)
  \ge
  \rho \left\{
    \tilde{\Lambda}
    -
    \tilde{\Lambda} J
    (J^\T \tilde{\Lambda} J)^{-1}
    J^\T \tilde{\Lambda}
  \right\}.
\end{align}
We prove this equality in the sequel.

Recall that we defined $\btheta_k^\T = (\alpha_k, \, \bbeta_k^\T)$.
Denote the information matrix with respect to
${(\btheta_1^\T, \, \ldots, \btheta_r^\T)}^\T$
under the DRM based on the first $r+1$ samples as $U_1$,
and that with respect to
${(\btheta_1^\T, \, \ldots, \btheta_m^\T)}^\T$
under the DRM based on all $m+1$ samples as $U_2$.
Let $U_{2,c}$ be the lower--right $(m-r)(d+1) \times (m-r)(d+1)$ block
of $U_2$.
Let $ \rho = \lim_{n \to \infty} (\sum_{k=0}^r n_k) / n$.

\begin{lemma} \label{lemma:estComp}
  Adopt the conditions of Theorem \ref{thm:DELRT}.
  We have:
  
  (1) $U_2/U_{2,c} \ge \rho U_1$. That is, $U_2/U_{2,c} - \rho U_1$ is
  positive semidefinite.

  (2)  $\Lambda/\Lambda_c \ge \rho \tilde{\Lambda}$.
\end{lemma}

\begin{lemma} \label{lemma:generalizedIneq}
Let $A$ be a $s \times s$ positive definite matrix and $B$ be a $s \times s$
positive semidefinite matrix. Also let $X$ and $Y$ be $s \times t$
matrices, and suppose the column space of $Y$ is contained in that of
$B$. Then
\begin{align*}
  (X+Y)^\T(A + B)^{-1}(X+Y)
  \le
  X^\T A^{-1} X
  +
  Y^\T B^\dagger Y
\end{align*}
where $B^\dagger$ is the Moore--Penrose pseudoinverse of $B$.
\end{lemma}

The proofs of the above two lemmas are given after the proof of
Theorem \ref{thm:pwComp}.

\begin{proof}[\sc{Proof of Theorem \ref{thm:pwComp}}]

  We prove equality \eqref{eq:pwCompToShow}.
  Define
  $
    M = \Lambda/\Lambda_c - \rho\tilde{\Lambda}
  $.
  In Lemma \ref{lemma:generalizedIneq}, let
  $A = \rho J^\T \tilde{\Lambda} J$, $B = J^\T M J$,
  $X = \rho J^\T \tilde{\Lambda}$, $Y = J^\T M$.
  Then $A + B = J^\T (\Lambda/\Lambda_c) J$ and
  $X + Y =  J^\T (\Lambda/\Lambda_c)$.
  Matrix $A$ is positive definite because $\tilde\Lambda$ is positive
  definite and $J$ is of full rank.
  $B$ is positive semidefinite because $M$ is positive semidefinite by
  Lemma \ref{lemma:estComp} (2).
  Moreover, it is easily seen that the column space of $Y$ is the same
  as that of $B$. Hence the conditions of Lemma
  \ref{lemma:generalizedIneq} are satisfied, and we have
  \begin{align*}
    &(\Lambda/\Lambda_c) J
    (J^\T (\Lambda/\Lambda_c) J)^{-1}
    J^\T (\Lambda/\Lambda_c)
    \le
    \rho \tilde{\Lambda} J
    (J^\T \tilde{\Lambda} J)^{-1}
    J^\T \tilde{\Lambda}
    +
    M J
    (J^\T M J)^\dagger
    J^\T M.
  \end{align*}
  The above inequality and
  $\Lambda/\Lambda_c = \rho\tilde{\Lambda} + M$
  imply that
  \begin{align*}
    &(\Lambda/\Lambda_c)
    -
    (\Lambda/\Lambda_c) J
    (J^\T (\Lambda/\Lambda_c) J)^{-1}
    J^\T (\Lambda/\Lambda_c)
    \\
    \ge
    &
    \rho \{
      \tilde{\Lambda}
      -
      \tilde{\Lambda} J
      (J^\T \tilde{\Lambda} J)^{-1}
      J^\T \tilde{\Lambda}
    \}
    +
    \{
      M
      -
      M J
      (J^\T M J)^\dagger
      J^\T M 
    \}.
  \end{align*}
  The term
  $
    M - M J (J^\T M J)^\dagger J^\T M 
  $
  is positive semidefinite because
  \begin{align*}
    M - M J (J^\T M J)^\dagger J^\T M
    =
    M^{1/2}
    \{ I - M^{1/2} J (J^\T M J)^\dagger J^\T M^{1/2} \}
    M^{1/2},
  \end{align*}
  and 
  $
    I
    -
    M^{1/2} J
    (J^\T M J)^\dagger
    J^\T M^{1/2}
  $
  is easily verified to be idempotent, hence positive semidefinite.
  Therefore inequality (\ref{eq:pwCompToShow}) holds and the claimed
  result is true.
\end{proof}

\begin{proof}[\sc{Proof of Lemma \ref{lemma:estComp} (1)}]

  We prove the result for $m=r+1$, namely $R_n^{(1)}$ uses all
  sample except for the last one. The general result is true by
  mathematical induction.

  Let $U_{2,a}$ be the upper--left $r(d+1) \times r(d+1)$ block, and
  $U_{2,b}$ be the upper--right $r(d+1) \times (m-r)(d+1)$ block, of
  $U_2$. Note that
  $U_2/U_{2,c} = U_{2,a} - U_{2,b}U_{2,c}^{-1}U_{2,b}^\T$,
  so to show the claimed result of
  $U_2/U_{2,c} \ge \rho \tilde{U_1}$, it suffices to show that
  \begin{align*} 
    (U_{2,a} - \rho U_1) - U_{2,b} U_{2,c}^{-1} U_{2,b}^\T
  \end{align*}
  is positive semidefinite.
  Notice that the above matrix
  is the Shur complement of
  \begin{align} \label{eq:infomatCompDexpression}
    D = 
    \left( \begin{array}{cc}
      U_{2,a} - \rho U_1 &U_{2,b}
      \\
      U_{2,b}^\T &U_{2,c}
    \end{array} \right)
    =
    U_2 - \diag(\rho U_1, \, 0).
  \end{align}
  By standard matrix theory,
  the positive semidefiniteness is implied by that of $D$.

  We now show $D$ is positive semidefinite. 
  We first give useful algebraic expressions for $U_2$ and $\rho U_1$.
  Notice that $(\btheta_1^\T, \, \ldots, \, \btheta_m^\T)$ is just
  permuted $\btheta^\T = (\balpha^\T, \bbeta^\T)$, the information
  matrix (\ref{eq:Uexpression}) of which helps us to obtain algebraic
  expressions for $U_1$ and $U_2$.
  Recall $\bQ(x) = {(1, \, \q^\T(x))}^\T$. 
  For $R_n^{(2)}$, we get
  \begin{align*}
  U_2
    = 
    \E_0 \big\{
    H(\btheta^*, x)  \otimes \{\bQ(x) \bQ^\T(x)\}
    \big\}.
  \end{align*}
  For $R_n^{(1)}$, we find
  \begin{align*}
    \rho U_1
    = 
    \E_0 \big\{
    H_r(\btheta^*, x)  \otimes \{\bQ(x) \bQ^\T(x)\}
    \big\},
  \end{align*}
  where $H_r(\btheta, x)$ is the $H$ matrix defined in
  (\ref{eq:def}) based on the first $r+1$ samples.
  Substituting the above expressions of $U_2$ and $\rho U_1$ into the
  expression \eqref{eq:infomatCompDexpression} of $D$, we get
  \begin{align*}
    D
    &=
     \rho_m \E_0 \big\{
      \{\bw(x) \bw^\T(x)\}
      \otimes \{\bQ(x) \bQ^\T(x)\}
    \big\},
  \end{align*}
  with
  \begin{align*}
    \bw(x)
    =
    \sqrt{\varphi_m(\btheta^*, \, x)}
    \left(\h_r^\T(\btheta^*, \, x), \,
    s_r(\btheta^*, \, x)\right)^\T /
    {\sqrt{s(\btheta^*, \, x) s_r(\btheta^*, \, x)}},
  \end{align*}
  where $\h_r(\btheta, \, x)$ and $s_r(\btheta, \, x)$ are the $\h$
  vector and $s$ defined in (\ref{eq:def}) based on the first $r+1$
  samples, respectively.
  Since $D$ is the expectation of the Kronecker product of two squares
  of vectors, it is positive semidefinite. This completes the proof.
\end{proof}

To prove Lemma \ref{lemma:estComp} (2),
partition $U_{\balpha \balpha}$, $U_{\balpha \bbeta}$ and
$U_{\bbeta \bbeta}$ as follows:
\begin{align*}
  U_{\balpha\balpha}
  &=
  \left( \begin{array}{cc}
    {U_{\balpha\balpha,a}}
    &
    {U_{\balpha\balpha,b}}
    \\
    {U^\T_{\balpha\balpha,b}}
    &
    {U_{\balpha\balpha,c}}
  \end{array} \right)
  , \  
  U_{\balpha\bbeta}
  =
  \left( \begin{array}{cc}
    {U_{\balpha\bbeta,a}}
    &
    {U_{\balpha\bbeta,b}}
    \\
    {U_{\balpha\bbeta,c}}
    &
    {U_{\balpha\bbeta,d}}
  \end{array} \right)
  ,
  \  
  U_{\bbeta\bbeta}
  =
  \left( \begin{array}{cc}
    {U_{\bbeta\bbeta,a}}
    &
    {U_{\bbeta\bbeta,b}}
    \\
    {U^\T_{\bbeta\bbeta,b}}
    &
    {U_{\bbeta\bbeta,c}}
  \end{array} \right),
\end{align*}
where $U_{\balpha\balpha,a}$, $U_{\balpha\bbeta,a}$ and $U_{\bbeta\bbeta,a}$
are the corrsponding upper--left $r \times r$, $r \times rd$ and $rd \times
rd$ blocks.

We also introduce an important property of the Schur complement.
Let
\begin{align*}
  M
  =
  \begin{pmatrix}
    \underset{s \times s}
    {A}
    &
    \underset{s \times t}
    {B}
    \\
    \underset{t \times s}
    {C}
    &
    \underset{t \times t}
    {D}
  \end{pmatrix}
  \ \text{ and } \  
  D =
  \begin{pmatrix}
    \underset{u \times u}
    {E}
    &
    \underset{u \times v}
    {F}
    \\
    \underset{v \times u}
    {G}
    &
    \underset{v \times v}
    {H}
  \end{pmatrix},
\end{align*}
where $u+v = t$.
Suppose $M$, $A$ and $D$ are nonsingular.
By Theorem 1.4 of \citet{Zhang2005}, the lower--right $u \times u$ block
of $M/H$ is just $D/H$, and
\begin{align} \label{eq:quotientFormula}
  M/D = (M/H)/(D/H).
\end{align}
The above equality is known as the \emph{quotient formula}.
Similar quotient formula holds for $M/A$.

\begin{proof}[\sc{Proof of Lemma \ref{lemma:estComp} (2)}]

  We first give an algebraic expression for $\Lambda / \Lambda_c$.
  Recall the definition
  $\Lambda = U_{\bbeta\bbeta} -
  U_{\bbeta\balpha} U_{\balpha\balpha}^{-1} U_{\balpha\bbeta}$,
  so
  \begin{align*}
    \Lambda = \Psi/U_{\balpha\balpha},
  \end{align*}
  where
  \begin{align*}
    \Psi
    =
    \begin{pmatrix}
      U_{\bbeta\bbeta} & U_{\bbeta\balpha}
      \\
      U_{\balpha\bbeta} & U_{\balpha\balpha}
    \end{pmatrix}.
  \end{align*}
  Let $\Psi_1$ be the lower--right
  $\{(m-r)d + m\} \times \{(m-r)d + m\}$ block of $\Psi$.
  Then $\Lambda_c$, the lower--right $(m-r)d$ $\times$ $(m-r)d$ block of
  \mbox{$\Lambda = \Psi/U_{\balpha\balpha}$}, satisfies 
  \begin{align*}
    \Lambda_c = \Psi_1/U_{\balpha\balpha}.
  \end{align*}
  Therefore
  \begin{align*}
    \Lambda/\Lambda_c
    =
    (\Psi/U_{\balpha\balpha})/(\Psi_1/U_{\balpha\balpha})
    =
    \Psi/\Psi_1,
  \end{align*}
  where the second equality above is by quotient formula
  (\ref{eq:quotientFormula}).

  It is easily seen that $\Psi/\Psi_1 = \Omega/\Omega_1$,
  where
  \begin{align*}
    \Omega
    =
    \left( \begin{array}{cc;{2pt/2pt}cc}
      U_{\bbeta\bbeta,a} & U_{\bbeta\balpha,a}
      &U_{\bbeta\bbeta,b} &U_{\bbeta\balpha,b}
      \\
      U_{\balpha\bbeta,a} &U_{\balpha\balpha,a}
      &U_{\balpha\bbeta,b} &U_{\balpha\balpha,b}
      \\
      \hdashline[2pt/2pt]
      U_{\bbeta\bbeta,b}^\T &U_{\bbeta\balpha,c}
      &U_{\bbeta\bbeta,c} &U_{\bbeta\balpha,d}
      \\
      U_{\balpha\bbeta,c} &U_{\balpha\balpha,b}^\T
      &U_{\balpha\bbeta,d} &U_{\balpha\balpha,c}
    \end{array} \right)
  \end{align*}
  and $\Omega_1$ is the lower--right block of $\Omega$ with the same
  size as that of $\Psi_1$.
  Thus we get
  \begin{align*}
    \Lambda/\Lambda_c = \Psi/\Psi_1 = \Omega/\Omega_1.
  \end{align*}

  Let $\Omega_2$ be the lower--right $(m-r)(d+1) \times (m-r)(d+1)$ block
  of $\Omega_1$.
  Matrix $\Omega_1/\Omega_2$
  is just the lower--right $r \times r$ block of $\Omega/\Omega_2$, and
  $
  \Omega/\Omega_1 = (\Omega/\Omega_2)/(\Omega_1/\Omega_2)
  $
  by quotient formula (\ref{eq:quotientFormula}).
  Hence, we finally get
  \begin{align*}
    \Lambda/\Lambda_c
    = \Omega/\Omega_1
    = (\Omega/\Omega_2)/(\Omega_1/\Omega_2).
  \end{align*}
  The above identity implies that our cliam of
  $\Lambda/\Lambda_c \ge \rho \tilde \Lambda$
  is equivalent to 
  \begin{align*} 
    (\Omega/\Omega_2)/(\Omega_1/\Omega_2) \ge \rho \tilde \Lambda.
  \end{align*}
  Further notice that
  $\tilde \Lambda = \check{U}/\tilde{U}_{\balpha\balpha}$,
  where
  \begin{align*}
    \check{U} =
    \begin{pmatrix}
      \tilde{U}_{\bbeta\bbeta}
      &
      \tilde{U}_{\bbeta\balpha}
      \\
      \tilde{U}_{\balpha\bbeta}
      &
      \tilde{U}_{\balpha\balpha}
    \end{pmatrix},
  \end{align*}
  so, the above inequality is equivalent to  
  \begin{align} \label{eq:quotientToShow}
    (\Omega/\Omega_2)/(\Omega_1/\Omega_2)
    \ge \rho (\check{U}/\tilde{U}_{\balpha\balpha}).
  \end{align}

  In the last step, we prove the above inequality
  (\ref{eq:quotientToShow}).
  By standard matrix theory, if matrices $M$ and $N$ are
  both positive definite and $M \ge N$, then the corresponding Schur
  complements satisfy the same inequality.
  Note that both $\Omega/\Omega_2$ and $\check{U}$ are positive
  definite, so to show (\ref{eq:quotientToShow}),
  it is enough to show that
  \begin{align*}
    \Omega/\Omega_2 \ge \rho \check{U}.
  \end{align*}
  Note that parameter
  $
    \bphi^\T
    =
    (\bbeta_1^\T, \, \ldots, \, \bbeta_r^\T, \,
     \alpha_1, \, \ldots, \, \alpha_r, \,
     \bbeta_{r+1}^\T, \, \ldots, \, \bbeta_m^\T, \,
     \alpha_{r+1}, \, \ldots, \, \alpha_m)
  $
  is just permuted $(\btheta_1^\T, \, \ldots, \, \btheta_m^\T)$, so
  the conculsion of Lemma \ref{lemma:estComp} (1) also applies to the
  information matrix with respect to $\bphi$.
  The information matrix with respect to $\bphi$ for $R_n^{(2)}$ is
  just $\Omega$, and its lower--right $(m-r)(d+1) \times (m-r)(d+1)$
  block is $\Omega_2$. 
  For $R_n^{(1)}$, the infromation matrix is just $\check{U}$.
  Thus by Lemma \ref{lemma:estComp} (1), we have
  $
    \Omega/\Omega_2 \ge \rho \check{U}
  $.
  The proof is complete.
\end{proof}

\begin{proof}[\sc{Proof of Lemma \ref{lemma:generalizedIneq}}]
  Notice that
  \begin{align*}
    \begin{pmatrix}
      A + B &X+Y
      \\
      (X+Y)^\T &X^\T A^{-1} X + Y^\T B^\dagger Y
    \end{pmatrix}
    =
    \begin{pmatrix}
      A &X
      \\
      X^\T &X^\T A^{-1} X
    \end{pmatrix}
    +
    \begin{pmatrix}
      B &Y
      \\
      Y^\T &Y^\T B^\dagger Y
    \end{pmatrix}.
  \end{align*}
  The first matrix on the RHS is positive semidefinite by Theorem 1.12
  of \citet{Zhang2005}, and since $Y$ is in the column space of $B$,
  the second matrix on the RHS is also positive semidefinite by
  Theorem 1.20 of \citet{Zhang2005}.
  Therefore the matrix on the left hand side (LHS) is positive
  semidefinite. Also note that $A+B$ is positive definite.
  Hence the Schur complement of the LHS with respect to its
  upper--left block $A+B$,
  \begin{align*}
      X^\T A^{-1} X + Y^\T B^\dagger Y
      -(X+Y)^\T (A+B)^{-1} (X+Y),
  \end{align*}
  must also be positive semidefinite.
  The claimed result then follows.
\end{proof}

%% file: tables.tex
\begin{table}[H]
\caption{
  Parameter values for power comparison under non--normal distributions
  (Section 5.3).
  $F_0$ remains unchanged across parameter settings 0--5.
}
\label{tab:pwcomp}
\tabcolsep 4pt
\scriptsize
\centering
\begin{tabular}{c|ccc|cc|cc|cc|cc}
  \multicolumn{12}{c}{$\Gamma(\lambda, \, \kappa)$: gamma distribution with shape
  $\lambda$ and rate $\kappa$;}
  \\
  \multicolumn{12}{c}{$LN(\mu, \, \sigma)$: log--normal distribution with mean $\mu$ and
  standard deviation $\sigma$ on log scale;}
  \\
  \multicolumn{12}{c}{$Pa(\gamma)$: Pareto distribution with shape $\gamma$ and common
  support of $x > 1$;}
  \\
  \multicolumn{12}{c}{$W(b)$: Weibull distribution with scale $b$ and common shape of
  $0.8$.}
  \\
  \hline
  \multicolumn{12}{c}{Parameter settings}
  \Tstrut
  \\
  $F_0$
  &
  &\multicolumn{2}{c}{1}
  &\multicolumn{2}{c}{2}
  &\multicolumn{2}{c}{3}
  &\multicolumn{2}{c}{4}
  &\multicolumn{2}{c}{5}
  \Bstrut
  \\
  \hline
  \Tstrut
  &
  &$\lambda$ &$\kappa$
  &$\lambda$ &$\kappa$
  &$\lambda$ &$\kappa$
  &$\lambda$ &$\kappa$
  &$\lambda$ &$\kappa$
  \\
  \multirow{4}{*}{$\Gamma(0.2, \, 0.8)$}
  &$F_1$:
  &0.18 &0.7
  &0.17 &0.6
  &0.16 &0.5
  &0.155 &0.45
  &0.14 &0.4
  \\
  &$F_2$:
  &0.22 &0.85
  &0.24 &0.95
  &0.255 &1.05
  &0.18 &0.7
  &0.17 &0.6
  \\
  &$F_3$:
  &0.23 &0.95
  &0.255 &1.2
  &0.275 &1.25
  &0.29 &1.4
  &0.33 &1.6
  \\
  &$F_4$:
  &0.24 &1.05
  &0.27 &1.3
  &0.29 &1.4
  &0.31 &1.55
  &0.35 &1.85
  \Bstrut
  \\
  \hline
  \Tstrut
  &
  &$\mu$ &$\sigma$
  &$\mu$ &$\sigma$
  &$\mu$ &$\sigma$
  &$\mu$ &$\sigma$
  &$\mu$ &$\sigma$
  \\
  \multirow{4}{*}{$LN(0, \, 1.5)$}
  &$F_1$:
  &0.44 &1.3
  &0.7  &1.2
  &0.9  &1.15
  &1    &1
  &1.2  &0.85
  \\
  &$F_2$:
  &0.22 &1.32
  &0.57 &1.30
  &0.62 &1.25
  &0.67 &1.20
  &0.87 &1
  \\
  &$F_3$:
  &0.18 &1.35
  &0.63 &1.33
  &0.73 &1.30
  &0.83 &1.28
  &0.85 &1.28
  \\
  &$F_4$:
  &0.37 &1.38
  &0.60 &1.35
  &0.70 &1.33
  &0.75 &1.32
  &0.95 &1.30
  \Bstrut
  \\
  \hline
  \Tstrut
  &
  &\multicolumn{2}{c|}{$\gamma$}
  &\multicolumn{2}{c|}{$\gamma$}
  &\multicolumn{2}{c|}{$\gamma$}
  &\multicolumn{2}{c|}{$\gamma$}
  &\multicolumn{2}{c}{$\alpha$}
  \\
  \multirow{4}{*}{$Pa(2)$}
  &$F_1$:
  &\multicolumn{2}{c|}{1.9}
  &\multicolumn{2}{c|}{1.85}
  &\multicolumn{2}{c|}{1.8}
  &\multicolumn{2}{c|}{1.75}
  &\multicolumn{2}{c}{1.7}
  \\
  &$F_2$:
  &\multicolumn{2}{c|}{2.1}
  &\multicolumn{2}{c|}{2.2}
  &\multicolumn{2}{c|}{2.3}
  &\multicolumn{2}{c|}{1.85}
  &\multicolumn{2}{c}{1.75}
  \\
  &$F_3$:
  &\multicolumn{2}{c|}{2.35}
  &\multicolumn{2}{c|}{2.55}
  &\multicolumn{2}{c|}{2.70}
  &\multicolumn{2}{c|}{2.85}
  &\multicolumn{2}{c}{3.25}
  \\
  &$F_4$:
  &\multicolumn{2}{c|}{2.5}
  &\multicolumn{2}{c|}{2.78}
  &\multicolumn{2}{c|}{2.98}
  &\multicolumn{2}{c|}{3.2}
  &\multicolumn{2}{c}{3.75}
  \Bstrut
  \\
  \hline
  \Tstrut
  &
  &\multicolumn{2}{c|}{$b$}
  &\multicolumn{2}{c|}{$b$}
  &\multicolumn{2}{c|}{$b$}
  &\multicolumn{2}{c|}{$b$}
  &\multicolumn{2}{c}{$b$}
  \\
  \multirow{4}{*}{$W(1)$}
  &$F_1$:
  &\multicolumn{2}{c|}{0.76}
  &\multicolumn{2}{c|}{0.65}
  &\multicolumn{2}{c|}{0.59}
  &\multicolumn{2}{c|}{0.53}
  &\multicolumn{2}{c}{0.42}
  \\
  &$F_2$:
  &\multicolumn{2}{c|}{1.2}
  &\multicolumn{2}{c|}{1.26}
  &\multicolumn{2}{c|}{1.31}
  &\multicolumn{2}{c|}{1.35}
  &\multicolumn{2}{c}{1.42}
  \\
  &$F_3$:
  &\multicolumn{2}{c|}{1.08}
  &\multicolumn{2}{c|}{1.05}
  &\multicolumn{2}{c|}{1.10}
  &\multicolumn{2}{c|}{1.12}
  &\multicolumn{2}{c}{1.14}
  \\
  &$F_3$:
  &\multicolumn{2}{c|}{0.90}
  &\multicolumn{2}{c|}{0.89}
  &\multicolumn{2}{c|}{0.85}
  &\multicolumn{2}{c|}{0.82}
  &\multicolumn{2}{c}{0.78}
  \Bstrut
  \\
  \hline
\end{tabular}
\end{table}

\begin{table}[H]
\caption{
  Parameter values for power comparison under misspecified DRMs
  (Section 5.4).
  $F_0$ remains unchanged across parameter settings 0--5.
}
\label{tab:pwcompmis}
\centering
\tabcolsep 4pt
\scriptsize
\begin{tabular}{c|ccc|cc|cc|cc|cc}
  \multicolumn{12}{c}{$W(a, \, b)$: Weibull distribution
  with shape $a$ and scale $b$.}
  \\
  \hline
  \multicolumn{12}{c}{Parameter settings}
  \Tstrut
  \\
  $F_0$
  &
  &\multicolumn{2}{c}{1}
  &\multicolumn{2}{c}{2}
  &\multicolumn{2}{c}{3}
  &\multicolumn{2}{c}{4}
  &\multicolumn{2}{c}{5}
  \Bstrut
  \\
  \hline
  \Tstrut
  &
  &$a$ &$b$
  &$a$ &$b$
  &$a$ &$b$
  &$a$ &$b$
  &$a$ &$b$
  \\
  \multirow{4}{*}{$W(1, \, 1)$}
  &$F_1$:
  &0.9   &0.95
  &0.85  &0.94
  &0.82  &0.92
  &0.79  &0.91
  &0.75  &0.88
  \\
  &$F_2$:
  &0.98  &0.98
  &0.96  &0.96
  &0.95  &0.95
  &0.94  &0.94
  &0.91  &0.92
  \\
  &$F_3$:
  &1.03  &1.04
  &1.05  &1.06
  &1.07  &1.07
  &1.09  &1.08
  &1.12  &1.12
  \\
  &$F_4$:
  &1.01  &0.95
  &1.02  &0.92
  &1.03  &0.90
  &1.05  &0.89
  &1.07  &0.85
  \Bstrut
  \\
  \hline
\end{tabular}
\normalsize
\end{table}

\begin{table}[H]
\caption{
  Parameter settings for power comparison of $R_n^{(1)}$ and
  $R_n^{(2)}$ (Section 5.5).
}
\label{tab:pwcomp1vs2}
\centering
\tabcolsep 4pt
\begingroup
    \fontsize{8pt}{9.6pt}\selectfont
\begin{tabular}{c|cccccc}
  \hline
  \multirow{4}{*}{Normal Case}
  &\multicolumn{6}{c}{Common parameter settings: $F_0: N(0, \, 1)$, $F_2: N(-1, \, 2)$}
  \\
  \cline{2-7}
  &\multicolumn{6}{c}{Parameter settings for $F_1$}
  \\
  &0
  &1
  &2
  &3
  &4
  &5
  \\
  &$N(1.5, \, 0.5)$
  &$N(1.57, \, 0.45)$
  &$N(1.58, \, 0.41)$
  &$N(1.6, \, 0.39)$
  &$N(1.62, \, 0.36)$
  &$N(1.64, \, 0.31)$
  \\
  \hline
  \hline
  \multirow{4}{*}{Gamma Case}
  &\multicolumn{6}{c}{Common parameter settings: $F_0: \Gamma(2, \, 1)$}
  \\
  \cline{2-7}
  &\multicolumn{6}{c}{Parameter settings for $F_1$}
  \\
  &0
  &1
  &2
  &3
  &4
  &5
  \\
  &$\Gamma(4, \, 3)$
  &$\Gamma(5.3, \, 4.3)$
  &$\Gamma(6.3, \, 5.3)$
  &$\Gamma(7.1, \, 6.1)$
  &$\Gamma(8.3, \, 7.3)$
  &$\Gamma(10, \, 9)$
  \\
  \hline
\end{tabular}
\endgroup
\end{table}